\documentclass[11pt]{article}
\usepackage{amsmath, amssymb,amsthm}
\usepackage{natbib}
\usepackage{a4wide}

\theoremstyle{plain}
\newtheorem{theorem}{Theorem}
\theoremstyle{plain}
\newtheorem{lem}{Lemma}
\newtheorem{proposition}{Proposition}

\theoremstyle{definition}
\newtheorem{defi}{Definition}
\newtheorem*{rem}{Remark}
\newtheorem{ex}{Example}

\usepackage{color}

\newcommand{\Levy}{L\'{e}vy }
\newcommand{\levy}{L\'{e}vy}
 \newcommand{\ca}{c\`adl\`ag}

\newcommand{\LSS}{\ensuremath{\mathcal{LSS}} }
\newcommand{\lss}{\ensuremath{\mathcal{LSS}}}

\newcommand{\V}{\ensuremath{\mathcal{VMLV}} }
\newcommand{\vv}{\ensuremath{\mathcal{VMLV}}}

\newcommand{\VB}{\ensuremath{\mathcal{VMBV}} }
\newcommand{\vvb}{\ensuremath{\mathcal{VMBV}}}

\newcommand{\R}{\ensuremath{\mathbb{R}}}

\newcommand{\E}{\ensuremath{\mathbb{E}}}

\newcommand{\Lp}{\ensuremath{L^{(p)}}}

\begin{document}
\title 
{On stochastic integration for volatility modulated L\'{e}vy--driven Volterra  processes}
\date{\today}
\author
{Ole E. Barndorff-Nielsen\footnote{Thiele Center, 
Department of Mathematical Sciences,
\&  CREATES,
 School of Economics and Management,
Aarhus University,
Ny Munkegade 118,
DK-8000 Aarhus C, Denmark,
{\tt oebn\@@imf.au.dk}}\\
Aarhus University
\and Fred Espen Benth\footnote{Centre of Mathematics for Applications,
 University of Oslo,
 P.O. Box 1053, Blindern,
 N-0316 Oslo, Norway
{\tt fredb\@@math.uio.no}}\\University of Oslo
\and
Jan Pedersen\footnote{Thiele Center, 
Department of Mathematical Sciences,
Aarhus University, Ny Munkegade 118,
DK-8000 Aarhus C, Denmark,
{\tt jan\@@imf.au.dk}}\\ Aarhus University
 \and Almut E.~D.~Veraart\footnote{Department of Mathematics, 
Imperial College London, 180 Queen's Gate,
SW7 2AZ London, 
 UK, {\tt a.veraart\@@imperial.ac.uk}}\\ Imperial College London \& CREATES}

\maketitle

\begin{abstract}
This papers develops a stochastic integration theory with respect to volatility modulated L\'{e}vy--driven Volterra ($\mathcal{VMLV}$) processes. It extends recent results in the literature to allow for stochastic volatility and pure jump processes in the integrator. The new integration operator is based on Malliavin calculus and describes an anticipative integral. 
 Fundamental properties of the integral are derived and important applications are given.

\end{abstract}
{\bf Keywords: } Volatility modulated Volterra process, L\'evy semistationary processes, stochastic integration, Skorohod integral, Malliavin calculus
\newpage
\section{Introduction}
\subsection{Motivation}
This paper establishes a stochastic integration theory for  volatility modulated Volterra ($\mathcal{VMLV}$)  processes which are  defined as 
\begin{align}\label{firstV}
X(t) = \int_0^t g(t,s ) \sigma(s) dL(s),
\end{align}   
where $g$ is a deterministic function, $\sigma$ is a predictable stochastic process and $L$ is a L\'{e}vy process. The precise regularity assumptions needed to ensure that the integral in (\ref{firstV}) is well-defined are given in the following Subsection.

Our research is closely related to the influential work by \cite{AMN} and \cite{D,Decreusefond2005}, who developed a  stochastic calculus for Gaussian processes. 
Our new contributions include in particular a stochastic integration theory with respect to Volterra processes which include a \emph{stochastic volatility/intermittency} component, which is denoted by $\sigma$ in \eqref{firstV}. As such we extend the results by \cite{AMN} in the Gaussian case to allow for volatility modulation. Moreover, we also develop a suitable stochastic integration theory for the case that $L$ is a pure jump L\'{e}vy process. Also  these new results hold in the presence of stochastic volatility. Furthermore, our proposed integral is of a slightly different form than the one suggested in \cite{AMN}.
In addition to developing the new integration type, we establish the main properties of our new integral, and we focus on a variety of interesting examples. E.g.~we study integration of deterministic kernel functions with respect to Volterra processes and we investigate the case when the integrand is itself a Volterra process, which leads to a representation result in terms of the second chaos expansion.

Due to the flexibility the class of \V offers, it can be applied to modelling a wide range of phenomena. E.g.~\V processes have recently been found to be suitable for e.g.~modelling turbulence in physics, see \cite{BNSch09}. Also, they can be used to model financial data, such as commodity spot prices, as recently shown in 
\cite{BNBV2010a}. In many applications, it is important to have a stochastic integration theory at hands, e.g.~in financial applications one typically models the value of a portfolio by a stochastic integral with respect to the price process. This paper aims at establishing the suitable integration theory which can then be used in various such applications.

The outline for the remaining part of the paper is as follows. First we review the technical assumptions needed to define integrals of the type  \eqref{firstV}. Next, since  Malliavin calculus both for Brownian motion  and for pure  jump L\'{e}vy processes is the key tool for establishing our integration theory w.r.t.~\V processes, we give a short review of main results from that calculus in Section 
\ref{SectMalliavin}.   Section 
\ref{SectVMLV} introduces the class of volatility modulated Volterra processes in more detail and describes relevant examples. 
The main contributions of the paper are contained in Section 
\ref{SectStochasticIntegration}, where we establish the integration concept with respect to volatility modulated Volterra processes -- both in the case when the process is driven by a Brownian motion and also in the case of a pure jump L\'{e}vy process. Next we derive the key properties of the new integral in Section 
\ref{RefPropApplications} and we focus on various applications of the new concept. Finally,  Section \ref{SectConcl} concludes and gives an outlook on future research.

\subsection{Preliminaries}\label{SectPrel}
Let  $(\Omega,\mathcal{F}, P)$ be a complete probability space equipped with a  filtration $\mathfrak{F}=(\mathcal{F}_t)_{t \geq 0}$ satisfying the usual conditions of right-continuity and completeness. 

We aim to develop a stochastic integration concept with respect to volatility modulated L\'{e}vy--driven ($\mathcal{VMLV}$) processes. In order to set the scene we first review the suitable integration concept for 
 proper  integrals with respect to a 
\Levy process $L$, i.e.\ integrals of the form 
\begin{equation}\label{J1}
\int_{0}^\infty  \phi(t)\,  d L(t),
\end{equation}
where $\phi=(\phi(t))_{t\in \R}$ is a   predictable integrand. Clearly, by replacing $\phi(t)$ by    $\phi(t) \mathbb{I}_{\{a<t\leq b\}}$ for $0\leq a <b<\infty$ we are able to treat integrals over the intervals 
 $(a,b]$ within this framework as well. Relevant references in this context include  
\cite{Kallenberg1975},
\cite{RajputRosinski},
\cite{Kallenberg1992},
\cite{D},
\citet[Theorem III.6.30
]{JS2003},
\cite{ChernyShiryaev2005},
\cite{DNMBOP},
\cite{P},
\cite{N},
\cite{BassePedersen2010}.

Let us recall some key properties of 
\Levy processes which we will apply throughout the paper.
\begin{defi}
 A stochastic process   $L=(L(t))_{t\geq 0}$ is a \emph{\Levy process} if it satisfies the following five conditions:
(1) $L(0)=0$ a.s.;
(2) $L$ has independent increments in the  sense that for any choice of  $n\geq 2$ and $t_1 <\cdots <t_n$ the increments $L(t_2) - L(t_1), \ldots, L({t_n}) - L(t_{n-1})$ are independent; 
(3) $L$ has stationary increments in the sense that
  for all $s<t$,  $L(t)-L(s)$ has the same law as $L(t-s)$;
  (4) $L$ is continuous in probability;
(5) $L$ has \ca\   paths.
\end{defi}
The \Levy process  $L$ has   a 
\levy--Khintchine representation, which is given by 
\begin{align}\label{LevyK}
\E(\exp(i u L(t))) = \exp(t \psi(u)), \;\; 
\psi(u) =  i\gamma u - \frac{c^2 u^2}{2}  +
\int_{\R}(e^{iuz}-1-iu\tau(z))\, \ell(dz),
\end{align}
where $\tau$ is a truncation function which is  assumed to behave like the identity function around the origin. More precisely, a standard choice would be
$\tau(z)= z\mathbb{I}_{\{|z| \leq  1\}}$. Further, $\gamma \in \mathbb{R}$, $c^2 \geq 0$ and
 $\ell$ denotes the \Levy measure of $L$ satisfying $\ell(\{0\})=0$ and $\int_{\mathbb{R}}(1\wedge |z|^2)\,\ell(dz) < \infty$. We write $(\gamma, c^2 , \ell)$ for the characteristic triplet of $L$.
Recall that we have the following \levy--It\^o--representation for $0\leq s<t$: 
\begin{align}\label{LevyIto}
L(t) - L(s) &= (t-s)\gamma  + c (B(t)- B(s)) + \int_s^t\int_{\R}\tau(z)\, \widetilde N(dz,du)
 \\
  &\ \ 
+ \int_s^t\int_{\R}(z-\tau(z))\, N(dz,du), \nonumber 
\end{align}
where $N$ denotes the corresponding Poisson random measure  and $\widetilde N(dz,du) = N(dz,du)-\ell(dz)du$ denotes the compensated Poisson random measure. Moreover, $B$ is a standard  Brownian motion which is independent of $N$. 
 
\subsubsection{Stochastic integration with respect to \Levy processes}
 Next we review briefly how the integral in \eqref{firstV} is defined.
 For \emph{deterministic} integrands one can apply the integration concept developed by \cite{RajputRosinski} for \Levy bases as integrators,  using the fact that 
\Levy processes can be directly linked to homogeneous \Levy bases, which are defined as homogeneous, infinitely divisible, independently scattered random measures, see  e.g.~\cite{RajputRosinski,Sato2004,BNSch07a}.

 The integration theory can be extended to allow for stochastic integrands as follows.
Recall that $\mathfrak{F}=(\mathcal{F}_t)_{t\in \R}$ denotes 
a filtration satisfying the usual conditions. Let $\mathcal{P}$ denote the predictable $\sigma$-field on $\R\times \Omega$, i.e.
\begin{align*}
\mathcal{P} = \sigma\left( (s, t]\times A: -\infty < s < t < \infty, A \in \mathcal{F}_s \right)
\end{align*}
which is  the smallest sigma-field generated by the left-continuous and adapted processes.

Assume that  
$L=(L_t)_{t\geq 0}$ is a 
\Levy process 
with respect to 
$\mathfrak{F}$ 
 with characteristic triplet $(\gamma, \sigma^2, \ell)$.  That is, 
\begin{equation}\label{J5}
 \mbox{for all $0\leq s<t$,   $L(t) - L(s)$ is independent of $\mathcal{F}_s$.}
 \end{equation}
As shown in \citet{BassePedersen2010}, see also \cite{ChernyShiryaev2005},  the integral 
$\int_{0}^{\infty}\phi(s)d L(s)$ exists 
if and only if $\phi$ is predictable and  the following three conditions are satisfied $a.s.$: 
 \begin{align}\begin{split}
  \label{J234}
\int_{0}^\infty \sigma^2\phi^2(s) \,d
    s<\infty, \qquad
 \int_{0}^\infty \int_{\R}
 \left(1\wedge (\phi(s) z)^2\right)\,  \ell(d z)\,d s<\infty,\\ 
 \int_{0}^\infty  \Big|\phi(s)\gamma+\int_{\R}
 \big[\tau(\phi(s)z)-\phi(s)
 \tau(z)\big]\,\ell(d  z)\Big|\,d s<\infty.
 \end{split}
 \end{align}

\subsubsection{Important examples}
We will now focus on some important
examples, 
 where the integrability conditions stated in \eqref{J234} simplify considerably. 
\begin{ex}\label{Jex1} 
If  in addition to \eqref{J5},  
 for all $s<t$,  $L(t) - L(s)$  has zero mean, is square-integrable, and  the predictable process $\phi$ satisfies $\E(\int_{0}^\infty \phi(s)^2ds)<\infty$, then \eqref{J1} exists, is square-integrable and $\int_{0}^t \phi_sd L_s$ is a square-integrable martingale up to infinity, see \citet{BassePedersen2010}  for more details. 
\end{ex} 
\begin{ex}
Suppose $L$ is a 
Brownian motion with drift. Then the characteristic triplet is given by $(\gamma, \sigma^2, 0)$ and condition  \eqref{J5} is satisfied. 
Then the integrability conditions for a predictable process $\phi$ simplify to 
\begin{align*}
\sigma^2 \int_{0}^t \phi(s)^2\,  ds < \infty, \quad 
\left|\gamma\right|\int_{0}^t  \left|\phi(s) \right|\, ds < \infty \ a.s.
\end{align*}
\end{ex}
 \begin{ex} 
See  \citet[Corollary 3.6]{ChernyShiryaev2005} for a detailed treatment of the 
  case when $(L_t)_{t \geq 0}$ is an $\alpha$--stable \Levy process.
\end{ex}

\section{A brief background on Malliavin Calculus}\label{SectMalliavin}
We briefly review the key results from Malliavin Calculus needed to establish our new stochastic integration theory. In the following, our presentation is
extracted from the monographs by~\cite{N}  and~\cite{DNOP}, where we  discuss the  Wiener and the pure-jump case separately.

\subsection{The Wiener case}
\label{subsect-wienermall} 
As before,  $(\Omega,\mathcal{F},P)$ denotes a complete probability space and we let
$(\mathcal{T},\mathcal{B},\mu)$ be a measure space with $\mu$ a
$\sigma$-finite measure without atoms. The space of square-integrable
functions on 
$(\mathcal{T},\mathcal{B},\mu)$ is  denoted by $L^2(\mathcal{T})$. In our
applications, $\mathcal{T}$ is typically the finite time interval
$\mathcal{T}=[0,T]$, for some given finite time horizon  $T<\infty$, equipped with the Lebesgue measure $\mu$ on the
Borel sets $\mathcal{B}$. 

We denote by $W$  the $L^2(\Omega)$ Gaussian measure on
$(\mathcal{T},\mathcal{B})$. That is, 
for every $A\in \mathcal{B}$ with $\mu(A)<\infty$, $W(A)$ is a
centered normally distributed random 
variable with variance $\mu(A)$. Furthermore, if $A,B\in\mathcal{B}$
with $\mu(A)<\infty,\mu(B)<\infty$,  
are disjoint, then $W(A)$ and $W(B)$ 
are independent.  Assume that $\mathcal{F}$ is the $\sigma$--field generated by $W(A)$ for $A\in \mathcal{B}$ with $\mu(A)<\infty$. 

To define multiple stochastic integrals with respect to $W$, consider
elementary functions 
$f\in L^2(\mathcal{T}^n,\mathcal{B}^n,\mu^n)$ given as
$$
f(t_1,\ldots,t_n)=\sum_{i_1,\cdots,i_n=1}^{m}a_{i_1\cdots
  i_n}\mathbf{1}_{A_{i_1}\times\cdots\times
  A_{i_n}}(t_1,\ldots,t_n)\,, 
$$  
for $A_{i_j}\in\mathcal{B}$ such that $\mu(A_{i_j})<\infty$ and with
the property that $a_{i_1\cdots i_n}=0$  
if any  two of the indices $i_1,\ldots,i_n$ are equal.  
Then, we define 
$$
I_n(f)=\sum_{i_1,\cdots,i_n=1}^{m}a_{i_1\cdots i_n}W(A_{i_1})\cdots
W(A_{i_n})\,. 
$$
The set of such elementary functions is dense in $L^2(\mathcal{T}^n)$
and the operator $I_n$ can 
be extended to a linear and continuous operator from
$L^2(\mathcal{T}^n)$ to $L^2(P)$. It has the  
properties that $I_n(f)=I_n(\widetilde{f})$ if $\widetilde{f}$ is the
symmetrization of $f$, and 
$$
\E\left[I_m(f)I_n(g)\right]=\left\{\begin{array}{ll} 0, & m\neq n\,,
    \\ n!\langle\widetilde{f},\widetilde{g}\rangle_n, & m=n\,,
  \end{array}\right.
$$
where we have denoted the inner product in $L^2(\mathcal{T}^n)$ by
$\langle\cdot,\cdot\rangle_n$. The notation  
$$
I_n(f)=\int_{\mathcal{T}^n}f(t_1,\ldots,t_n)W(dt_1)\cdots W(dt_n)\,,
$$
will be frequently used.

The famous Wiener-It\^o chaos decomposition (\citet[Theorem 1.1.2]{N}) holds for the space $L^2(P)$: 
\begin{theorem}
If $X\in L^2(P)$, then there exists a unique sequence 
$\{f_n\}_{n=0}^{\infty}$ of symmetric 
functions where $f_n\in L^2(\mathcal{T}^n)$ such that
$$
X=\sum_{n=0}^{\infty}I_n(f_n)\,.
$$
Moreover,
$$
\|X\|_2^2:=\mathbb{E}[X^2]=\sum_{n=0}^{\infty}n!|f_n|^2_{n}\,.
$$
where $|\cdot|_n$ is the norm in $L^2(\mathcal{T}^n)$. 
\end{theorem}  

The Malliavin derivative of a random variable $X\in L^2(P)$ can be
characterized as an 
operation acting on the chaos. The domain of the derivative can also
be characterized by 
the chaos functions. The following discussion is taken from
 \citet[Proposition 1.2.1]{N}.
Suppose $X\in  L^2(P)$ with chaos expansion 
$$
X=\sum_{n=0}^{\infty}I_n(f_n)
$$
is satisfying 
$$
\sum_{n=1}^{\infty}n n!|f_n|^2_n<\infty\,.
$$
Then we say that $X\in\mathbb{D}^{1,2}$, the domain of the Malliavin
derivative $D_t$, and 
it holds that 
\begin{equation}
D_tX=\sum_{n=1}^{\infty}nI_{n-1}(f_n(\cdot,t))\,.
\end{equation}
The Malliavin derivative is thus the operation where we fix the last
coordinate of the  
chaos function $f_n$, and moves it to a chaos of order one $n-1$. It
corresponds to the annihilation 
operator in quantum physics. A direct computation shows that
\begin{equation}
\mathbb{E}\left[\int_{\mathcal{T}}(D_tX)^2\,\mu(dt)\right]=\sum_{n=1}^{\infty}n
n!|f_n|^2_n\,. 
\end{equation} 
In particular, this proves that $D_t$ is a linear operator from
$\mathbb{D}^{1,2}$ into  
$L^2(\Omega\times \mathcal{T})$. 

The Skorohod integral is introduced as the adjoint operator of $D_t$ in the 
following way: denote by $\text{Dom}(\delta)$ the set of all processes
$Y$ in $L^2(\Omega\times\mathcal{T})$ where
\begin{equation}
\vert\mathbb{E}\left[\int_{\mathcal{T}}Y(t)D_tX\,\mu(dt)\right]\vert\leq
c\|X\|_2\,, 
\end{equation}
for all $X\in\mathbb{D}^{1,2}$, and $c>0$ being a constant depending
on $Y$. Then, 
for $Y\in\text{Dom}(\delta)$, we define the Skorohod integral of $Y$, denoted by 
\begin{equation}
\int_{\mathcal{T}}Y(t)\delta W(t)\,,
\end{equation}
as the unique element of $L^2(P)$ characterised by
\begin{equation}
\mathbb{E}\left[X\int_{\mathcal{T}}Y(t)\delta W(t)\right]=
\mathbb{E}\left[\int_{\mathcal{T}}Y(t)D_tX\,\mu(dt)\right]\,,
\end{equation}
for all $X\in\mathbb{D}^{1,2}$.

One may view the Skorohod integral as a creation operator on chaos, as
the following shows.  
It holds by the Wiener-It\^o chaos expansion that for any $Y\in
L^2(\Omega\times\mathcal{T})$  
there exists a sequence of symmetric functions $f_n(\cdot,t)$ in
$L^2(\mathcal{T}^n)$ such that 
$$
Y(t)=\sum_{n=0}^{\infty}I_n(f_n(\cdot,t))\,.
$$   
In this representation, $t$ is only considered as a parameter of the function 
$f_n(\cdot,t)$ on $\mathcal{T}^n$. Letting $\widetilde{f}_n\in
L^2(\mathcal{T}^{n+1})$ be  
the symmetrization of $f_n(\cdot,t)$, it holds that 
\begin{equation}
\int_{\mathcal{T}}Y(t)\delta
W(t)=\sum_{n=0}^{\infty}I_{n+1}(\widetilde{f}_n)\,. 
\end{equation}
Moreover, the space of Skorohod integrable processes
$\text{Dom}(\delta)$ is characterized  
by processes satisfying
$$
\sum_{n=0}^{\infty}(n+1)!|\widetilde{f}_n|_{n+1}^2<\infty\,,
$$
where this sum is equal to the variance of the Skorohod integral.

The Skorohod integral is a generalization of It\^o integration in the
case of $\mathcal{T}=[0,T]$ 
and $\mu$ being the Lebesgue measure on the Borel $\sigma$-algebra
$\mathcal{B}$. 
We have from  \citet[Proposition 1.3.4]{N} that 
$$
\int_0^TY(t)\delta W(t)=\int_0^TY(t)\,dW(t),
$$
whenever $Y$ is It\^o integrable. The essential extension offered by
the Skorohod integral is that one can define the stochastic integral
of $Y$ with 
respect to Brownian motion without restricting
to integrands $Y$ being adapted to the filtration generated by the
Brownian motion.  
Sometimes one refers to the
Skorohod integral as an {\it anticipative} stochastic integral for this reason.

Skorohod integration and Malliavin differentiation have several
interesting properties, and 
we present the ones which are relevant for our analysis of stochastic
integration. First,  
the following \lq\lq fundamental theorem of calculus\rq\rq\ holds:
\begin{proposition}
Suppose that $Y\in\text{Dom}(\delta)\cap\mathbb{D}^{1,2}$ and 
$\int_{\mathcal{T}}Y(t)\delta W(t)\in\mathbb{D}^{1,2}$. If the process
$s\mapsto D_t Y(s)$ for  
$s\in\mathcal{T}$ is Skorohod integrable for almost every
$t\in\mathcal{T}$, then 
\begin{equation}
D_t\left(\int_{\mathcal{T}}Y(s)\delta
  W(s)\right)=Y(t)+\int_{\mathcal{T}}D_tY(s)\delta W(s)\,, 
\end{equation} 
for almost every $t\in\mathcal{T}$.
\end{proposition}
The next result, which is an integration by parts formula,  will be  the key to define the stochastic integrals we are
interested in this  
paper:
\begin{proposition}[Integration by parts formula]\label{prop2}
Suppose that $Y\in\text{Dom}(\delta)$ and $X\in\mathbb{D}^{1,2}$. If
$XY\in\text{Dom}(\delta)$, 
then
\begin{equation}
\int_{\mathcal{T}}XY(t)\delta
W(t)= X\int_{\mathcal{T}}Y(t)\delta W(t)-\int_{\mathcal{T}}Y(t)D_tX\,\mu(dt). 
\end{equation}
\end{proposition}
This result is proved and discussed in \citet[page 40]{N}. 

\subsection{The pure-jump L\'evy process case}
\label{subsect-levymall}

Our exposition of the Malliavin Calculus for L\'evy processes is based on the 
monograph~\cite{DNOP}. Here, the space
$(\mathcal{T},\mathcal{B},\mu)$ is 
explicitly chosen to be $\mathcal{T}=[0,T]$ with $\mu$ being the
Lebesgue measure on the 
Borel subsets $\mathcal{B}$ of $[0,T]$. We work under precisely the same assumption in the following review.

Suppose that $\Lp$ is a pure-jump L\'evy process which is square
integrable and with 
mean zero. Let 
$N(dz,dt)$ denote  the Poisson random measure associated to $\Lp$ on
$[0,T]\times\mathbb{R}_0$, with 
$\mathbb{R}_0:=\mathbb{R}\backslash\{0\}$. Moreover, we denote by
$\ell(dz)$ the  
L\'evy measure of $\Lp$. 
 In terms of its L\'evy-It\^o representation, we have
\begin{equation}\label{lipois}
\Lp(t)=\int_0^t\int_{\mathbb{R}_0}z\widetilde{N}(dz,ds)\,,
\end{equation}
where $\widetilde{N}(dt,dz)$ is the compensated
Poisson random measure.  
For $0\leq
t\leq T$, we denote by  
$\mathcal{F}_t$ the $\sigma$-algebra generated by $\Lp(s)$ for $s\leq
t$, augmented by the  sets  
of $P$-zero probability.  Assume that $\mathcal{F}=\mathcal{F}_T$.

Consider now the space $L^2(([0,T]\times\mathbb{R}_0)^n,
(dt\times\ell(dz))^n)$, that is, measurable  
real-valued functions $f$ on $([0,T]\times\mathbb{R}_0)^n$ such that
$$
|f|_{n,\ell}^2:=\int_{([0,T]\times\R_0)^n}f^2(t_1,z_1,t_2,z_2,\ldots,t_n,z_n)\,dt_1\, 
\ell(dz_1)\,dt_2\,\ell(dz_2)\cdots\,dt_n\,\ell(dz_n)<\infty\,.
$$ 
Furthermore, we denote by $\widetilde{f}$ the symmetrization of $f$ in
the pairs of 
$n$ variables $(t_1,z_1),\ldots,(t_n,z_n)$. Any function which is
symmetric in these pairs 
is called a symmetric function from now on.

The $n$-fold iterated stochastic integral of a symmetric function
$f\in L^2(([0,T]\times\mathbb{R}_0)^n)$ 
with respect to the compensated Poisson random measure is now defined to be
\begin{equation}
I_n(f):=n!\int_0^T\int_{\mathbb{R}_0}\cdots\int_0^{t_2-}\int_{\mathbb{R}_0}f(t_1,z_1,\ldots,t_n,z_n) 
\widetilde{N}(dz_1,dt_1)\cdots\widetilde{N}(dz_n,dt_n)\,.
\end{equation}
We have the Wiener-It\^o chaos expansion for all square integrable
random variables: 
\begin{theorem}
If $X\in L^2(P)$ is $\mathcal{F}_T$-measurable, then there exists a
unique sequence 
of symmetric functions $\{f_n\}_{n=0}^{\infty}$ where $f_n\in
L^2(([0,T]\times\mathbb{R}_0)^n)$ 
such that 
$$
X=\sum_{n=0}^{\infty}I_n(f_n)\,.
$$
Moreover,
$$
\|X\|^2_2=\sum_{n=0}^{\infty}n!|f_n|^2_{n,\ell}\,.
$$
\end{theorem}
As in the Wiener case, we introduce a Malliavin derivative as an annihiliation 
operator on chaos. Since the kernel functions $f_n$ now are functions of 
pairs $(t,z)$, it is natural to introduce a derivative operator
$D_{t,z}$ parametrized over 
the pair $(t,z)$ rather than $t$ as was the case for Brownian
motion. We say that 
$X\in\mathbb{D}^{1,2}$, the domain of the Malliavin derivative
$D_{t,z}$, whenever 
$$
\sum_{n=1}^{\infty}nn!|f_n|_{n,\ell}^2<\infty\,.
$$ 
In this case, we define
\begin{equation}
D_{t,z}X=\sum_{n=1}^{\infty}nI_{n-1}(f_n(\cdot,t,z))\,.
\end{equation}
Again, as in the Wiener case, we can introduce a Skorohod integral as
the adjoint of the  
Malliavin derivative. We say that the space-time random field
$Y\in L^2([0,T]\times\mathbb{R}_0\times\Omega)$ is Skorohod integrable
(that is, belongs to 
$\text{Dom}(\delta)$) if
$$
\vert\mathbb{E}\left[\int_0^T\int_{\mathbb{R}_0}Y(t,z)D_{t,z}X\,\ell(dz)\,dt\right]\vert 
\leq c\|X\|_2\,,
$$ 
for all $X\in\mathbb{D}^{1,2}$, and $c$ being a constant only
depending on $Y$. If  
$Y\in\text{Dom}(\delta)$, then the Skorohod integral of $Y$, denoted by 
\begin{equation}
\int_0^T\int_{\mathbb{R}_0}Y(t,z)\widetilde{N}(\delta z,\delta t)
\end{equation}
is defined as the unique element in $L^2(P)$ characterised by
\begin{equation}
\mathbb{E}\left[X\int_0^T\int_{\mathbb{R}_0}Y(t,z)\widetilde{N}(\delta
  z,\delta t)\right] 
=\mathbb{E}\left[\int_0^T\int_{\mathbb{R}_0}Y(t,z)D_{t,z}X\,\ell(dz)\,dt\right]\,, 
\end{equation} 
for all $X\in\mathbb{D}^{1,2}$. 

The Skorohod integral can be viewed as a creation
operator on the chaos: it can be shown that $Y\in\text{Dom}(\delta)$
if and only if 
\begin{equation}
\sum_{n=0}^{\infty}(n+1)!|\widetilde{f}_n|_{n+1,\ell}^2<\infty\,,
\end{equation}
where $\widetilde{f}_n\in L^2(([0,T]\times\mathbb{R}_0)^{n+1})$ is the
symmetrization 
of the $n$th chaos kernel function $f_n(\cdot,t,z)$ of $Y$. The
Skorohod integral becomes 
\begin{equation}
\int_0^T\int_{\mathbb{R}_0}Y(t,z)\widetilde{N}(\delta z,\delta t)=
\sum_{n=0}^{\infty}I_{n+1}(\widetilde{f}_n)\,.
\end{equation} 
We see that the definitions of the Malliavin derivative and Skorohod
integration 
is completely analogous in the Wiener and pure-jump L\'evy cases. 

We now move on to some properties of the Skorohod integral, which
turn  out to be 
slightly different from  the Wiener case. We start with the
\lq\lq fundamental theorem of calculus\rq\rq: 
\begin{proposition}
Suppose that $Y\in\text{Dom}(\delta)\cap\mathbb{D}^{1,2}$. If the random field
$(s,y)\mapsto D_{t,z}Y(s,y)$ is Skorohod integrable, for almost every
$(t,z)$, then 
\begin{equation}
D_{t,z}\int_0^T\int_{\mathbb{R}_0}Y(s,y)\widetilde{N}(\delta y,\delta
s)=Y(t,z)+ 
\int_0^T\int_{\mathbb{R}_0}D_{t,z}Y(s,y)\widetilde{N}(\delta y,\delta s)\,.
\end{equation}
\end{proposition}
The next result, which is an integration by parts formula,  is particularly useful in our stochastic integral definition
for \V\  processes.
\begin{proposition}[Integration by parts formula]\label{prop4}
Suppose that $Y\in\text{Dom}(\delta)$ and $X\in\mathbb{D}^{1,2}$. If the random
field $Y(t,z)(X+D_{t,z}X)$ is Skorohod integrable, then
\begin{align}
X\int_0^T\int_{\mathbb{R}_0}Y(t,z)\widetilde{N}(\delta z,\delta
t)&=\int_0^T\int_{\mathbb{R}_0} 
Y(t,z)(X+D_{t,z}X)\widetilde{N}(\delta z,\delta t) \nonumber \\
&\qquad+\int_0^T\int_{\mathbb{R}_0}Y(t,z)D_{t,z}X\,\ell(dz)\,dt\,.
\end{align}
\end{proposition}
We observe that this relationship is not completely analogous to the
Gaussian case, as we  
have the additional  term $Y(t,z)D_{t,z}X$ inside the Skorohod integral.

\section{Volatility modulated \levy--driven Volterra processes}\label{SectVMLV}
Let us now define the class of \emph{volatility modulated \levy--driven Volterra} (\V) processes.
As before, we denote by $L=(L_t)_{t\geq 0}$   a   \Levy process. 
A \V  process $X=(X(t))_{t \geq 0}$, is defined  by
 \begin{equation}
\label{def-vmlv}
X(t)=\int_{0}^tg(t,s)\sigma(s)\,dL(s), \qquad t \geq 0,
\end{equation}
 where $g$ is a real--valued measurable function defined on
the space $\mathbb{R}^2_>:=\{(t,s)\in\mathbb{R}^2\,\vert\,t> s\geq 0\}$ and 
$\sigma= (\sigma(t))_{t \geq 0}$ is an $\mathfrak{F}$--adapted, predictable stochastic process.  
Note that generally $g(t,s)$ is not defined for $t=s$. 

We always assume that the integrand $s\mapsto g(t,s)\sigma(s)$ satisfies the integrability conditions
(\ref{J234}) $a.s.$ for all $t \in \R$, making   $X$   well-defined and $\mathfrak{F}$--adapted. 

Note that in financial modelling, $\sigma$ is usually referred to as the {\it stochastic volatility}, while in applications to turbulence, it is referred to as the {\it stochastic intermittency}. In some  applications, the stochastic processes  $\sigma$ and $L$ are assumed to be independent.

\subsection{Examples of the kernel function}
Let us study some relevant examples of \V processes, which are often used in applications.

An important type of  kernel function is the so--called
\emph{shift-kernel}, which is defined as  $g(t,s) = g(t-s)$ for all $t > s\geq 0$. 
This choice is motivated from the class 
of  \emph{\Levy semistationary} (\lss) processes. 
In the Brownian case such processes have been used for modelling turbulence, see \cite{BNSch09}, and the extension to more general \Levy processes has been introduced in \cite{BNBV2010a} in the context of modelling energy spot prices.

\begin{ex}
The \LSS processes encompass many existing dynamical models. For example, by choosing $g(t,s)=\exp(-\alpha(t-s))$ for $s \leq t$ with $\alpha > 0$ and  
$\sigma(s)=1$, we recover the 
solution of the Ornstein--Uhlenbeck (OU) process
$$
dX(t)=-\alpha X(t)\,dt+dL(t)\,.
$$
\end{ex}
\begin{ex}
A generalization of OU--processes and still within the class of $\mathcal{LSS}$ processes, is the class of CARMA($p,q$)--processes.
A CARMA process is the continuous--time analogue of an ARMA time
series, see \cite{Brockwell2001a, Brockwell2001b} for details, and it
is defined as follows. Suppose that for $p, q \in \mathbb{N}_0$ with $p>q$,   
$
X(t)=\mathbf{b}'\mathbf{V}(t)$,
where $\mathbf{b}\in\R^p$ and $\mathbf{V}(t)$ is a $p$-dimensional OU process of the form
\begin{equation}
\label{multi-OU-carma}
d\mathbf{V}(t)={\bf A}\mathbf{V}(t) dt+\mathbf{e}_pd L(t),
\end{equation}
with 
$$
{\bf A}=\left[\begin{array}{cc} \mathrm{0} & {\bf I}_{p-1} \\ -\alpha_p & -\alpha_{p-1}\cdots-\alpha_1\end{array}\right]\,. 
$$
Here we use the notation ${\bf I}_{p-1}$ for the $(p-1)\times (p-1)$--identity matrix, $\mathbf{e}_p$ is the $p$th coordinate vector (where the first $p-1$ entries are zero and the $p$th entry is 1) and
$\mathbf{b}'=[b_0, b_{1},\ldots,b_{p-1}]$ is the transpose of $\mathbf{b}$, with $b_q=1$ and $b_j=0$ for $q<j<p$. 
\cite{Brock} shows  that if all the eigenvalues of ${\bf A}$ have negative real parts, then
$\mathbf{V}(t)$ defined as
$$
\mathbf{V}(t)=\int_{-\infty}^t e^{{\bf A}(t-s)}\mathbf{e}_p\,dL(s)\,,
$$
is the (strictly) stationary solution of \eqref{multi-OU-carma}. Furthermore,  
\begin{align}\label{CARMA}
X(t)= \mathbf{b}'{\bf V}(t) = \int_{-\infty}^t\mathbf{b}'e^{{\bf A}(t-s)}\mathbf{e}_p\,dL(s)\,,
\end{align}
is a CARMA($p$, $q$) process.  
Hence, a CARMA process is essentially an \LSS process with constant volatility and kernel function specified as
$g(t,s)=\mathbf{b}'\exp({\bf A}(t-s))\mathbf{e}_p$. 
\end{ex}

\begin{ex}\label{TurbulenceEx}
In applications to turbulence, one often works with the following  function: 
\begin{equation}\label{toyeq}
g(t,s)=g(t-s)= (t-s)^{\nu-1}\exp(-\lambda(t-s))\,,
\end{equation}
with $\nu>\frac{1}{2}$ and $\lambda>0$, see \cite{BNSchmiegel2008}. This choice leads to a well-defined \LSS process, which is stationary (provided the stochastic volatility component is stationary). However, notice that $g(\cdot)$ has a singularity   at zero when $\nu \in (\frac{1}{2},1)$.   
\end{ex}

\begin{ex} Important examples within our modelling framework are also fractional Brownian motion and  fractional \Levy processes. To see this, suppose that 
  $\sigma\equiv 1$. Choosing
\begin{equation}\label{fbm1}
g(t,s)=c(H)(t-s)^{H-1/2}+c(H)\left(\frac12-H\right)\int_s^t(u-s)^{H-3/2}\left(1-(s/u)^{1/2-H}\right)\,du\,,
\end{equation}
with
$$
c(H)=\sqrt{\frac{2H\Gamma(\frac32-H)}{\Gamma(H+\frac12)\Gamma(2-2H)}}\,,
$$
and $H\in(0,1)$, we recover fractional Brownian motion (see 
\cite{AMN}, page 798):
\begin{equation*}
X(t)= \int_0^t g(t,s)\, dB(s),\quad t\geq0,
\end{equation*}
where $B$ is a standard Brownian motion. 
Recall the following alternative representation of fractional Brownian motion (cf.\ e.g.\ \cite{Nualart06}):
\begin{equation}\label{fbmrep}
X(t)= \int_{-\infty}^\infty g(t,s)\, dB(s),  
\end{equation}
where 
\begin{equation}\label{fbm2}
g(t,s)= \frac{1}{C_1(H)}\left( (t-s)^{H-1/2}_+ - (-s)_+^{H-1/2}\right), \quad x_+:= \max\{x,0\}, 
\end{equation}
and
\begin{equation*}
C_1(H)=\left( \int_0^\infty \left((1+s)^{H-1/2}- s^{H-1/2}\right)^2\, ds +\frac{1}{2H}\right)^{1/2}\, .
\end{equation*}
Replacing $B$ in \eqref{fbmrep} by a zero mean square-integrable two-sided pure-jump L\'evy process $ L(t)$ and assuming $H\in (1/2,1)$ we get a so-called fractional L\'evy process as defined by \citet{Marquardt}.
\end{ex}
\subsection{Examples of the stochastic volatility/intermittency process}
So far, we only mentioned that we allow for stochastic volatility/intermittency. Let us briefly point out  relevant specifications of such processes. 
We often model $\sigma^2$ as a $\mathcal{VMLV}$ process itself, i.e.~let 
\begin{align}\label{Modelsigma}
\sigma^2_t = \int_{0}^t i(t,s) d L_{\sigma}(s),
\end{align}
 where $i$ is a real-valued measurable function defined on $\mathbb{R}_{>}^2$ which is integrable with respect to the  L\'{e}vy subordinator  $L_{\sigma}$. 
  Moreover, if one wants to ensure that the volatility process is stationary, one can work with the specification 
\begin{align}\label{Modelsigma2}
\sigma^2_t = \int_{-\infty}^t i(t-s) d L_{\sigma}(s),
\end{align}
where  $L_{\sigma}$ is a two--sided \Levy process and the integration is defined as in \cite{RajputRosinski}. 
A simple example which involves dependence between $\sigma$ and $L$ is the following one. 
\begin{ex}
Let $L$ denote a two--sided L\'{e}vy process with inverse Gaussian (IG) marginal distribution. Further choose
$L_{\sigma}= L$ and $i(t,s) = \exp(-\beta(t-s))$ for $\beta > 0$. Then $\sigma^2$ is an Ornstein-Uhlenbeck process driven by an inverse Gaussian L\'{e}vy process, i.e.~an IG-OU process.  
\end{ex}


\subsection{Semimartingale conditions}
 It is important to note  that generally \V processes are not semimartingales. 
E.g.~consider the case of Example \ref{TurbulenceEx}, which is relevant in the context of modelling turbulence; there the kernel function was chosen as 
$g(t,s)= (t-s)^{\nu-1}\exp(-\lambda(t-s))$, 
with $\nu>\frac{1}{2}$ and $\lambda>0$. When $\nu \in (\frac{1}{2}, 1)$ or $\nu \in(1, \frac{3}{2}]$, the corresponding \V process is not a semimartingale, and these cases are of primary interest in the turbulence context, see \cite{BNSch09}.

However, 
as long as $g$ is non-singular and satisfies a smoothness condition, we can show that $X$ is a semimartingale. Indeed, we  have the following result, which is closely related to e.g.\ Corollary~4.9 in \citet{Protter85} and Theorem 4.6 in~\citet{Andreas}: 
\begin{proposition}\label{semima}
Suppose $X$ is a $\mathcal{VMLV}$ process defined as in \eqref{def-vmlv}, and
assume that  $t\mapsto \sigma(t)$ is locally bounded pathwise $a.s.$ Suppose $g(t,s)$ is defined for all $0\leq s \leq t$
and  that there is a bi-measurable  function $\phi(t,s)$ such that 
\begin{equation}\label{integ1}
g(t,s)= g(s,s) +\int_s^t\phi(v,s)\,dv \quad \mbox{for all }0\leq s\leq t,
\end{equation}
where $\int_0^t (g(s,s))^2<\infty$ for all $t>0$ and 
\begin{equation}\label{integ2}
\int_0^t\int_0^u(\phi(u,s))^2\,ds\, du<\infty, \quad \mbox{for all } t>0. 
\end{equation} 
Then, $(X(t))_{t \geq 0}$ is a semimartingale with decomposition
\begin{equation}
dX(t)=g(t,t)\sigma(t)\,d L(t)+\int_{0}^t\phi(t,s)\sigma(s)\,dL(s)\,dt\,, 
\end{equation}
that is
\begin{equation}\label{decx}
X(t)= \int_0^tg(s,s)\sigma(s)\, d L(s) +\int_0^t \int_0^v \phi(v,s)\sigma(s)\,d L(s)\,dv,\, \qquad t \geq 0. 
\end{equation}
\end{proposition}
Note that \eqref{integ1} essentially is equivalent to the assumption that $g(s,s)$ exists and $t\mapsto g(t,s)$ is differentiable for $t>s$ with $\phi(t,s)= \frac{\partial g}{\partial t}(t,s)$. 
\begin{proof}
We have 
\begin{align*}
X(t)&= \int_0^t g(t,s) \sigma(s)\, d L(s)= \int_0^tg(s,s)\sigma(s)\, dL(s) + \int_0^t\left[\int_s^t\phi(v,s)\, dv\right]\sigma(s)\, d L(s)\\
&=\int_0^tg(s,s)\sigma(s)\, d L(s)+ \int_0^t\left[\int_0^v\phi(v,s)\sigma(s)\, d L(s)\right]\,  dv,
\end{align*}
where we have applied   Fubini's theorem, see \citet[Theorem IV.65]{P}, which is justified by the assumptions on $\phi$ and $\sigma$.   
\end{proof}
See also \cite{BassePedersen2009} and 
  \cite{Basse2010} for related results.

\begin{ex}
Consider the  example from turbulence where we choose $g(t,s) = (t-s)^{\nu-1} e^{_-\lambda(t-s)}$ with $\lambda > 0$ and $\nu > \frac12$. 
Note that assuming that $\nu>\frac{3}{2}$  implies that $g(s,s)=0$.  The integrability conditions \eqref{integ1} and \eqref{integ2} are then satisfied with $\phi(t,s) =\frac{\partial g}{\partial t}(t,s)$ for $t>s$. Also note that $t\mapsto g(t,s)$ is differentiable at $t=s$  when $\nu>2$.
According to Proposition~\ref{semima}, the semimartingale cases are $\nu=1$ and $\nu >\frac{3}{2}$.   In fact, when $\nu>\frac{3}{2}$ the process is even of bounded variation. 
See \cite{BNSch09,BNCP11} for more details.   
\end{ex}

\section{Stochastic integration}\label{SectStochasticIntegration}

In this section we are concerned with stochastic integration with respect to a \V process $X$, that is, to define
\begin{equation}
\label{stochint}
\int_{0}^tY(s)\,dX(s), \qquad t \geq 0,
\end{equation}
for a class of stochastic processes $Y(s)\,,s\in\mathbb{R}$. The integrands $Y$ may be  \V processes themselves. We are concerned with having a definition which may facilitate the integration with respect to possibly singular \V processes, in the sense of kernel functions  $g(t,s)$ which may not be defined for $t=s$.

Since at  present we only have a theory of integration with respect to Brownian and pure-jump \Levy processes separately we consider  two cases: 
First, we study the case when $X$ is a 
\emph{volatility modulated Brownian Volterra} (\vvb) process,
\begin{equation}
\label{def-bss}
X(t)=\int_{0}^tg(t,s)\sigma(s)\,dB(s),  
\end{equation}
where $B$ is a standard Brownian motion. 
Next, we treat  the {\it pure-jump} \V case
\begin{equation}
\label{def-lss}
X(t)=\int_{0}^tg(t,s)\sigma(s)\,d\Lp(s),
\end{equation} 
where $\Lp$ is a zero-mean square integrable pure-jump \Levy process.

\subsection{The \VB case}
We start with the case of a volatility modulated Brownian driven Volterra process.
Throughout this Section, we will be working under the following assumption.
\begin{description}
\item[Assumption A]
The process $B=(B(t))_{t\geq 0}$   is a  standard Brownian motion. For $t\geq 0$ let $\mathcal{F}_t$ be  the $\sigma$-field generated by $B(s)$ with $s\leq t$.  The process  $(\sigma(t))_{t\geq 0}$ is  $(\mathcal{F}_t)_{t\geq 0}$-predictable and 
\begin{equation}\label{inteq}
\E\left(\int_0^t(g(t,s)\sigma(s))^2\, ds\right)<\infty,
\end{equation}
for all $t\geq 0$. Finally, $\mathcal{F}$ is the $\sigma$-field generated by $B$. 
\end{description}
When Assumption A is satisfied   $X(t)$ defined as in \eqref{def-bss} 
is a well-defined square-integrable zero mean \vvb\ process.

In order to motivate our definition of stochastic integration, we start out with a \emph{heuristic} derivation. First, supposing  $t\mapsto X(t)$ and $s\mapsto Y(s)$ are differentiable, integration by parts yields 
\begin{align*}
\int_{0}^tY(s)\,dX(s)&=\int_{0}^tY(u)\frac{dX(u)}{du}\,du
=\left\{Y(s)X(s)\right\}_{s=0}^t-\int_{0}^{t}\frac{d Y(u)}{d u}X(u)\,du\\
&=Y(t)X(t)-\int_{0}^t\frac{d Y(u)}{d u}\int_{0}^ug(u,s)\sigma(s)\,dB(s)\,du\,.
\end{align*}
In the integral on the right-hand side, we use Proposition~\ref{prop2}
to get 
\begin{align*}
\int_{0}^tY(s)\,dX(s)&=Y(t)X(t)-\int_{0}^t\int_{0}^u\frac{d Y(u)}{d u}g(u,s)\sigma(s)\delta B(s)\,du \\
&\qquad-\int_{0}^t\int_{0}^u D_s\left\{\frac{d Y(u)}{d u}\right\}g(u,s)\sigma(s)\,ds\,du\,.
\end{align*}
Here, $\delta B$ denotes stochastic integration with respect to Brownian motion in the Skorohod sense, while $D_s$ is the
Malliavin derivative. Applying a stochastic Fubini theorem, 
 it holds that 
\begin{align*}
\int_{0}^tY(s)\,dX(s)&=Y(t)X(t)-\int_{0}^t\left(\int_s^tg(u,s)\frac{d Y(u)}{d u}\,du\right)\sigma(s)\,\delta B(s) \\
&\qquad-\int_{0}^tD_s\left\{\int_s^tg(u,s)\frac{d Y(u)}{d u}\,du\right\}\sigma(s)\,ds\,.
\end{align*}
Here, we used that the Malliavin derivative is linear.
On the first term on the right--hand side we again apply Proposition~\ref{prop2}
to reach
\begin{align*}
Y(t)X(t)&=Y(t)\int_{0}^tg(t,s)\sigma(s)\,dB(s)\\
&=\int_{0}^tY(t)g(t,s)\sigma(s)\,\delta B(s)+\int_{0}^t
D_s\left\{Y(t)\right\}g(t,s)\sigma(s)\,ds\,.
\end{align*}
Thus, collecting terms we reach the expression
\begin{align*}
\int_{0}^tY(s)\,dX(s)&=\int_{0}^t\left\{Y(t)g(t,s)-\int_s^tg(u,s)\frac{d Y(u)}{du}\,du\right\}\sigma(s)\,\delta B(s) \\
&\qquad+\int_{0}^tD_s\left\{Y(t)g(t,s)-\int_s^tg(u,s)\frac{d Y(u)}{du}\,du\right\}\sigma(s)\,ds\,.
\end{align*}
Observe that after doing an integration by parts, we find
\begin{align*}
Y(t)g(t,s)&-\int_s^tg(u,s)\frac{d Y(u)}{du}\,du \\
&=Y(s)g(s,s)+\int_s^tY(u)\frac{\partial g(u,s)}{\partial u}\,du \\
&=Y(s)g(t,s)+\int_s^t\left(Y(u)-Y(s)\right)\frac{\partial g(u,s)}{\partial u}\,du\,.
\end{align*}
Here we have assumed that $g$ is differentiable with respect to its first argument. Note that in the last equality we have reorganized the integral such that the term $g(s,s)$ is not appearing, thus ensuring that we can include integration for \V processes which may be singular at $g(s,s)$. 

The above derivation motivates the introduction of the operator
\begin{equation}
\label{K_g-def}
\mathcal{K}_g(h)(t,s)=h(s)g(t,s)+\int_s^t\left(h(u)-h(s)\right)g(du,s)\,,
\end{equation}
whenever integration with respect to $g(du,s)$ makes sense as a Lebesgue-Stieltjes integral.  That is, for any $s\geq 0$, the mapping $t\mapsto g(t,s)$ is of bounded variation for $t$ in  any bounded interval $[u,v]$ with $u>s$.  We remark that precisely the same operator  $\mathcal{K}_g$ appears in 
\citet[p.~770]{AMN} in their definition of stochastic  integration for Gaussian processes. In the particular case of fractional Brownian motion  this  operator    has also been very useful, cf.\ e.g.\ \citet{BHOZ}.
Note that the definition does not require the existence of $g(s,s)$. Moreover,  the mapping   $t\mapsto g(t,s)$ can be  of unbounded variation on   $(s,v)$ for any $v>s$. 

\begin{ex}
Note that the example of key relevance in modelling turbulence is the choice of the kernel function given in 
equation  \eqref{toyeq}. The above operator can handle this case  for any $\nu>1/2$. To see that note that  $t\mapsto g(t,s)$ is of bounded variation on $[u,v]$ but  unbounded variation for $t\in (s,v)$ for  $s<u<v$ if $\nu\in (\frac{1}{2}, 1)$; further  $g(s,s)=1$  and $t\mapsto g(t,s)$ is of finite variation on any interval if $\nu=1$,  and $g(s,s)=0$  and $t\mapsto g(t,s)$ is of bounded variation on any interval if $\nu>1$. In the case of a fractional Brownian motion where $g$ is given by \eqref{fbm1} or \eqref{fbm2}, $g(s,s)$ does not exist and $t\mapsto g(t,s)$ is of unbounded variation on $(s,v)$ for any $v>0$ when $H\in (0,\frac{1}{2})$, $g(s,s)=1$ and $t\mapsto g(s,t)$ is of bounded variation on compacts when $H=\frac{1}{2}$ and $g(s,s)=0$ and  $t\mapsto g(s,t)$ is of bounded variation on compacts when $H>\frac{1}{2}$.  
\end{ex}

If $g(s,s)<\infty$, we can redefine $\mathcal{K}_g(h)$  as long as $u\mapsto h(u)$ is integrable with respect to $g(du,s)$ in the Lebesgue-Stieltjes sense as
\begin{equation}
\mathcal{K}_g(h)(t,s)=h(s)g(s,s)+\int_s^th(u)\,g(du,s)\,.
\end{equation}
There is also  the case of $g(u,s)$ being absolutely continuous with respect to Lebesgue measure, yielding 
\begin{equation}
\mathcal{K}_g(h)(t,s)=h(s)g(t,s)+\int_s^t(h(u)-h(s))\frac{\partial g}{\partial u}(u,s)\,du\,
\end{equation}
with $\partial g/\partial u$ being the Radon--Nikodym derivative of $g$ with respect to the first variable.

Based on the above heuristic arguments we are now ready to define $\int_0^tY(s)\, dX(s)$. 

\begin{defi}\label{defbm}
Let  Assumption A  be  satisfied and  let $X$ be given by \eqref{def-bss}.
  Assume that for  $s\in \R_+$ the mapping $u\mapsto g(u,s)$ is of bounded variation on  $[u,v]$ for all  $0\leq s<u<v<\infty$. Fix  $t>0$. 

   We say that a  process
$$
s\mapsto Y(s)\,,
$$ 
for $s\in[0,t]$ belongs to $\mathcal{I}^X(0,t)$ if the following conditions are satisfied:
\begin{enumerate}
\item  For     $s\in [0,t]$ the process  
$(Y(u)-Y(s))_{u\in (s,t]}$  is integrable with respect to $g(du,s)$ $a.s.$
\item The mapping 
$$
s\mapsto\mathcal{K}_g(Y)(t,s)\sigma(s)\mathbb{I}_{[0,t]}(s)\,,
$$
is Skorohod integrable  with respect to $B$.

\item $\mathcal{K}_g(Y)(t,s)$ is Malliavin differentiable with respect to $D_s$ for $s\in[0,t]$, with
$$
s\mapsto D_s\left\{\mathcal{K}_g(Y)(t,s)\right\}\sigma(s)
$$
being Lebesgue integrable on $[0,t]$.
\end{enumerate}
Suppose that $s\mapsto Y(s)$ for $s\in[0,t]$ belongs to $\mathcal{I}^X(0,t)$. Then the stochastic integral of $Y$ with respect to $X$ is defined as
\begin{align}\label{defbr}
\int_{0}^tY(s)\,dX(s)&=\int_0^t\mathcal{K}_g(Y)(t,s)\sigma(s)\,\delta B(s) +\int_{0}^tD_s\{\mathcal{K}_g(Y)(t,s)\}\sigma(s)\,ds\,. 
\end{align}
\end{defi}
The integral on the right-hand side of \eqref{defbr} is defined in the sense of Subsection~\ref{subsect-wienermall}, where as always we use the convention $\int_0^t:= \int_{\R} \mathbb{I}_{[0,t]}$.

There is no conditions of adaptedness on the stochastic process $Y(s)$. Our definition gives an anticipative stochastic integral. Note that the definition does not make use of $g(t,t)$ explicitly, and thus we can include kernel functions which are singular, such as the  choice of kernel function  often used in modelling turbulence, see \eqref{toyeq}, as  kernels in the Brownian semistationary process. 

Note here that the integral we define in \eqref{defbr} generally differs from the integral studied by \cite{AMN} since we have the additional integral with respect to the Lebesgue measure.

\subsection{The pure--jump \V case}
Now we turn our 
 attention to pure--jump \V  processes. 
In what comes, we will work under the following assumption.
\begin{description}
\item[Assumption B] 
Let $T<\infty$ denote a finite time horizon. 
The process $\Lp=(\Lp(t))_{t\in [0,T]}$ is a  square integrable pure-jump \Levy process with zero mean and, for $0\leq t\leq T$,   $\mathcal{F}_t$ is the $\sigma$-field generated by $\Lp(s)$ with $s\leq t$.  The process  $(\sigma(t))_{t\in [0,T]}$ is  $(\mathcal{F}_t)_{t\in [0,T]}$-predictable and \eqref{inteq} is satisfied for all $t\in [0,T]$. 
 Finally, $\mathcal{F}=\mathcal{F}_T$. 
\end{description}

Let us work under Assumption B and let  $(X(t))_{t\in [0,T]}$ be the pure--jump \V process defined by \eqref{def-lss}.  By decomposing $\Lp$ as in \eqref{lipois}, $X$ is given by 
$$
X(t)=\int_{0}^tg(t,s)\sigma(s)\,d\Lp(s)=\int_{0}^t\int_{\mathbb{R}}zg(t,s)\sigma(s)\,\widetilde{N}(dz,ds).
$$
 Our goal is the same as in the \VB case, to define 
$$
\int_{0}^tY(s)\,dX(s), 
$$
for a class of integrands $Y(s)$. The extension to the L\'evy case rests on the Malliavin Calculus for L\'evy processes. We refer to 
\cite{DNMBOP} for the development of such a calculus, cf.\ also Subsection~\ref{subsect-levymall}. The main ingredients of interest to us are the extension of Skorohod integration,  the Malliavin derivative and the integration by parts formula for constants. 

We start first with a heuristic derivation following the \VB case: We find
$$
\int_{0}^tY(s)\,dX(s)=Y(t)\int_{0}^tg(t,s)\sigma(s)\,d\Lp(s)-\int_{0}^t\frac{d Y(u)}{du}\int_{0}^ug(u,s)\sigma(s)\,d\Lp(s)\,du\,.
$$ 
By applying the integration by parts formula in Proposition~\ref{prop4},
we find for the first term 
\begin{align*}
Y(t)\int_{0}^t\int_{\mathbb{R}}zg(t,s)\sigma(s)\,\widetilde{N}(dz,ds) 
&=\int_{0}^t\int_{\mathbb{R}}zg(t,s)\sigma(s)\,\widetilde{N}(dz,ds) \\
&\qquad +\int_{0}^t\int_{\mathbb{R}}zg(t,s)\left(Y(t)+D_{s,z}\left\{Y(t)\right\}\right)\sigma(s)\,\widetilde{N}(\delta z,\delta s) \\
&\qquad+\int_{0}^t\int_{\mathbb{R}}zg(t,s)D_{s,z}\left\{Y(t)\right\}\sigma(s)\,\ell(dz)\,ds\,.
\end{align*}
Here, $D_{s,z}$ is the Malliavin derivative with respect to the Poisson measure $N$, and $\widetilde{N}(\delta s,dz)$ denotes the Skorohod integral with respect to the compensated Poisson random measure. Note  the additional term $D_{s,z}\left\{Y(t)\right\}$ in the Skorohod integral which is not present in the corresponding term in the \VB case. Similarly, we find for the second term
\begin{align*}
\int_{0}^t\frac{dY(u)}{du}&\int_{0}^u\int_{\mathbb{R}}zg(u,s)\sigma(s)\,\widetilde{N}(dz,ds)\,du \\
&=\int_{0}^t\int_{0}^u\int_{\mathbb{R}}zg(u,s)\left(\frac{dY(u)}{du}+D_{s,z}\left\{\frac{dY(u)}{du}\right\}\right)\sigma(s)\,
\widetilde{N}(\delta z,\delta s)\,du \\
&\qquad+\int_{0}^t\int_{0}^u\int_{\mathbb{R}}zg(u,s)D_{s,z}\left\{\frac{dY(u)}{du}\right\}\sigma(s)\ell(dz)\,ds\,du \\
&=\int_{0}^t\int_{\mathbb{R}}\int_s^tzg(u,s)\left(\frac{dY(u)}{du}+D_{s,z}\left\{\frac{dY(u)}{du}\right\}\right)\,du\sigma(s)\,\widetilde{N}(\delta z,\delta s) \\
&\qquad+\int_{0}^t\int_{\mathbb{R}}\int_s^tzg(u,s)D_{s,z}\left\{\frac{dY(u)}{du}\right\}\,du\,\ell(dz)\sigma(s)\,ds\,,
\end{align*}
where we have used the stochastic Fubini theorem in the last equality. Next, by the linearity of the Malliavin derivative, we find
\begin{align*}
\int_{0}^t\frac{dY(u)}{du}&\int_{0}^u\int_{\mathbb{R}}zg(u,s)\sigma(s)\,\widetilde{N}(dz,ds)\,du \\
&=\int_{0}^t\int_{\mathbb{R}}z\left(\int_s^tg(u,s)\frac{dY(u)}{du}\,du+D_{s,z}\left\{\int_s^tg(u,s)\frac{dY(u)}{du}\,du\right\}\right)\sigma(s)\,\widetilde{N}(\delta z,\delta s) \\
&\qquad+\int_{0}^t\int_{\mathbb{R}}zD_{s,z}\left\{\int_s^tg(u,s)\frac{dY(u)}{du}\,du\right\}\,\ell(dz)\sigma(s)\,ds\,,
\end{align*}
As a last step, by again doing an integration by parts as in the \VB case, and collecting terms, we reach
\begin{align}\begin{split}\label{YdXJ}
\int_{0}^tY(s)\,dX(s)&=\int_{0}^t\int_{\mathbb{R}}z\left(\mathcal{K}_g(Y)(t,s)+D_{s,z}\left\{\mathcal{K}_g(Y)(t,s)\right\}\right)\sigma(s)\,\widetilde{N}(\delta z,\delta s) \\
&\qquad+\int_{0}^t\int_{\mathbb{R}}z\left(D_{s,z}\left\{\mathcal{K}_g(Y)(t,s)\right\}\right)\,\ell(dz)\sigma(s)\,ds\,.
\end{split}
\end{align}
This relation involves a Malliavin derivative for L\'evy processes and a Skorohod integral for Poisson random measures, along with the operator $\mathcal{K}_g$. 

Next, we introduce the class of integrands $\widetilde{\mathcal{I}}^X(0,t)$ for stochastic integration with respect to pure--jump \V processes. 
\begin{defi}\label{defpj}Let  Assumption B be  satisfied and  let $X$ be given by \eqref{def-lss}.  Assume that for  $s\in [0.T] $ the mapping $g(u,s)$,  defined for $u\in (s,T]$,  is of bounded variation. Fix  $t\in]0,T]$. 

A process
$$
s\mapsto Y(s)\,,
$$ 
for $s\in[0,t]$ belongs to $\widetilde{\mathcal{I}}^X(0,t)$ if the following conditions are satisfied:
\begin{enumerate}
\item For     $s\in [0,t]$ the process  
$(Y(u)-Y(s))_{u\in (s,t]}$  is integrable with respect to $g(du,s)$ $a.s.$

\item $\mathcal{K}_g(Y)(t,s)$ is Malliavin differentiable with respect to $D_{s,z}$ for $(s,z)\in[0,t]\times\mathbb{R}$, with
$$
(s,z)\mapsto zD_{s,z}\left\{\mathcal{K}_g(Y)(t,s)\right\}\sigma(s),
$$
being $\ell(dz)\,ds$-integrable.

\item The mapping
$$
(s,z)\mapsto z\left(\mathcal{K}_g(Y)(t,s)+D_{s,z}\left\{\mathcal{K}_g(Y)(t,s)\right\}\right)\sigma(s)\mathbb{I}_{[0,t]}(s),
$$
is Skorohod integrable on $[0,T]\times\mathbb{R}$ with respect to $\widetilde{N}(ds,dz)$. 
\end{enumerate}
Suppose that $s\mapsto Y(s)$ for $s\in[0,t]$ belongs to $\widetilde {\mathcal{I}}^X(0,t)$. Then the stochastic integral of $Y$ with respect to $X$ is defined as
\begin{align*}
\int_{0}^tY(s)\,dX(s)&= \int_0^t\int_{\mathbb{R}}z\left(\mathcal{K}_g(Y)(t,s)+D_{s,z}\left\{\mathcal{K}_g(Y)(t,s)\right\}\right)\sigma(s)\widetilde{N}(\delta z,\delta s) \\
&\qquad+\int_{0}^t\int_{\mathbb{R}}zD_{s,z}\left\{\mathcal{K}_g(Y)(t,s)\right\}\,\ell(dz)\sigma(s)\,ds\ .
\end{align*}
\end{defi}
The first integral on the right-hand side  is defined in the sense of Subsection~\ref{subsect-levymall}, where $\int_0^t:= \int_0^T\mathbb{I}_{[0,t]}$.

\subsection{Further remarks}
In order to verify conditions (1) in   Definition~\ref{defbm} and \ref{defpj} we must study existence of   $\mathcal{K}_g(h)$. 
Recall that 
\begin{equation*}
\mathcal{K}_g(h)(t,s)= h(s)g(t,s) + \int_s^t (h(u)-h(s))\, g(du,s). 
\end{equation*}
Thus, $\mathcal{K}_g(h)(t,s)$ exists if e.g.\ $h(t)= z \mathbb{I}_{\{a\leq t\leq b\}} $ 
for some constants $z\in \R$ and $a<b$. 
 If  the measure   $g(du,s)$  is finite on any finite interval   then $\mathcal{K}_g(h)(t,s)$ exists if $h$ is bounded on compacts and in particular if $u\mapsto h(u)$ is  \ca\ or c\`agl\`ad. This covers the turbulence example in equation \eqref{toyeq}   with $\nu\geq 1$ and fractional Brownian motion  with $H\geq 1/2$. In general, however, some kind of smoothness of $h(t)$  near $t=s$ is needed  to ensure existence of $\mathcal{K}_g(h)(t,s)$. For example, assume $u\mapsto g(u,s)$ is $C^1$ with $\frac{d}{du} g(u,s)\sim (u-s)^\alpha$ as $u\downarrow s$ where $\alpha>-2$. If  there is a constant $K$  such that $|h(u) -h(s)|\leq K(u-s)$ for all $u>s$, then   $\mathcal{K}_g(h)(t,s)$ exists.  This covers the  turbulence example  (equation \eqref{toyeq}), where    $\alpha= \nu-2$, as well as fractional Brownian motion, equations \eqref{fbm1} or \eqref{fbm2},  where  $\alpha= H-3/2$.

The following Lemma gives useful information in the  $\mathcal{LSS}$ case. 
\begin{lem}
Suppose $g$ is a shift-kernel, i.e.~$g(t,s):=g(t-s)$. 
If $f$ is Lebesgue-Stieltjes integrable on $[s,t]$ with respect to $g(du,s)$, then $f(s+\cdot)$ is Lebesgue-Stieltjes integrable on $[0,t-s]$ with respect to $g(du)$ and
\begin{equation}\label{stateq}
\int_s^tf(u)\,g(du,s)=\int_0^{t-s}f(s+u)\,g(du)\,.
\end{equation}
\end{lem}
\begin{proof}
Note that if $E\in\mathcal{B}([s,t])$, then $F:=E-s\in\mathcal{B}([0,t-s]) $, where $\mathcal{B}(A)$ is the Borel $\sigma$--algebra on $A$, an interval of $\mathbb{R}$. Since for any 
 $(a,b]\in\mathcal{B}([s,t])$ we have that $(a-s,b-s]\in\mathcal{B}([0,t-s])$,  it holds that 
$$
g((a,b],s)=g(b-s)-g(a-s)=g((a-s,b-s])\,.
$$
implying 
  translation invariance of the Lebesgue-Stieltjes measures $g(d \cdot, s)$ and $g(d \cdot)$. Hence, by definition of the Lebesgue  integral the Lemma holds. 
\end{proof}
The Lemma gives us the form of $\mathcal{K}_g$ in the case of $g$ being a shift-kernel:
\begin{equation}
\mathcal{K}_g(h)(t,s)=h(s)g(t-s)+\int_0^{t-s}\left(h(s+u)-h(s)\right)g(du)\,.
\end{equation}
Note that $\mathcal{K}_g(h)$ does not become a shift kernel. 

\section{Some properties and applications}\label{RefPropApplications}
In this section we will either work under Assumption A, in which case  $X$ is  given by \eqref{def-bss}, or under Assumption B, in which case $X$ is given by \eqref{def-lss}. Let $L=B$ when working under Assumption A and $L= \Lp$ when working under Assumption B.

\subsection{Fundamental properties} 
Let $0<t_1 <t_2$ and assume $s\mapsto Y(s)$ is in $\mathcal{I}^X(0,t_1)$  (resp.\ in $\widetilde {\mathcal{I}}^X(0,t_1)$).  
Then $s\mapsto Y(s)\mathbb{I}_{\{s\leq t_1\}}$ is in  $\mathcal{I}^X(0,t_2)$ (resp.\ in $\widetilde {\mathcal{I}}^X(0,t_2)$), and we have the usual rule 
\begin{equation*}
\int_0^{t_2}Y(s)\mathbb{I}_{\{s\leq t_1\}}\, dX(s)= \int_0^{t_1}Y(s)\, dX(s).
\end{equation*}

The definition of a stochastic integral  is linear: If $Y$ and $Z$ are two processes in $\mathcal{I}^X(0,t)$  (resp.\ in  $\widetilde {\mathcal{I}}^X(0,t)$) and  $a,b$ two constants, then we easily see by linearity of the Malliavin derivatives, Skorohod integration and Lebesgue integration that 
$$
\int_{0}^t(aY(s)+bZ(s))\,dX(s)=a\int_{0}^tY(s)\,dX(s)+b\int_{0}^tZ(s)\,dX(s)\,.
$$
This in particular implies that if $0\leq u<v\leq t$ and   $s\mapsto Y(s)$ is in $\mathcal{I}^X(0,u)$ and in $\mathcal{I}^X(0,v)$ (resp.\  in $\widetilde {\mathcal{I}}^X(0,u)$ and in $\widetilde {\mathcal{I}}^X(0,v)$), then $s\mapsto Y(s)\mathbb{I}_{\{u<s\leq v\}}$ is in 
$\mathcal{I}^X(0,t)$  (resp.\ in $\widetilde {\mathcal{I}}^X(0,t)$) and 
\begin{equation*}
\int_0^t Y(s)\mathbb{I}_{\{u<s\leq v\}}\, dX(s)= \int_0^vY(s)\, dX(s) - \int_0^uY(s)\, dX(s)\, .
\end{equation*}
Since it is obvious that $\mathcal{K}_g(1)(t,s) = g(t,s)$ for $s<t$ and $D_s \mathcal{K}_g(Y)(t,s) =\mathcal{D}_{s,z} \mathcal{K}_g(Y)(t,s)=0$ when $Y$ is deterministic,  it follows that 
\begin{equation}\label{ex1}
\int_0^t 1\, dX(s) = X(t)\quad
\text{and, more generally,} \quad \int_0^t\mathbb{I}_{\{u<s\leq v\}}\, dX(s) = X(v)- X(u).
\end{equation} 
The following integration--by--parts formulae  hold:
\begin{proposition}\label{ibp}
(1) Let Assumption A be satisfied and $X$ be  the $\mathcal{VMBV}$ process given by \eqref{def-bss}. 
Let $s\mapsto Y(s)$ be an element in $\mathcal{I}^X(0,t)$. Suppose $Z$ is a bounded random variable such that the process $s\mapsto ZY(s)\in\mathcal{I}^X(0,t)$. Then it holds that 
\begin{align}\label{browcase}
\int_{0}^tZY(s)\,dX(s)&=Z\int_0^tY(s)\,dX(s). 
\end{align}
(2)  Let Assumption B be satisfied and  $X$ be  the $\mathcal{VMLV}$ process given by \eqref{def-lss}. Let $s\mapsto Y(s)$ be an element in $\widetilde{\mathcal{I}}^X(0,t)$. Suppose $Z$ is a bounded random variable such that the process $s\mapsto ZY(s)\in\widetilde{\mathcal{I}}^X(0,t)$. Then it holds that
\begin{align}\begin{split}\label{purecase}
\int_{0}^tZY(s)\,dX(s)&=Z\int_0^tY(s)\,dX(s) \\
&\quad-\int_0^t\int_{\mathbb{R}}zD_{s,z}\left\{\mathcal{K}_g(Y)(t,s)\right\}D_{s,z}\left\{Z\right\}\sigma(s)\,\left(\widetilde{N}(\delta z,\delta s)+\ell(dz)\,ds\right)\,. \end{split} 
\end{align}

\end{proposition}
\begin{proof}
Without loss of generality, we suppose for notational simplicity that $\sigma(s)=1$. 
Note that  
\begin{equation}\label{extra}
\mathcal{K}_g(ZY)(t,s)=Z\mathcal{K}_g(Y)(t,s)\,.
\end{equation}
Proof of (1). By the product rule for the Malliavin derivative
\begin{equation}\label{prule}
D_s\left\{Z\mathcal{K}_g(Y)(t,s)\right\}=D_s\left\{Z\right\}\mathcal{K}_g(Y)(t,s)+ZD_s\left\{\mathcal{K}_g(Y)(t,s)\right\}\,.
\end{equation}
Hence, by \eqref{extra} we have 
\begin{align*}
\int_0^tZY(s)\,dX(s)&=\int_0^tZ\mathcal{K}_g(Y)(t,s)\,\delta B(s)\\
&\qquad+\int_0^tZD_s\left\{\mathcal{K}_g(Y)(t,s)\right\}+D_s\left\{Z\right\}\mathcal{K}_g(Y)(t,s)\,ds.
\end{align*} 
Applying  Proposition~\ref{prop2} this gives 
\begin{align*}
\int_0^tZY(s)\,dX(s)&=Z\left(\int_0^t\mathcal{K}_g(Y)(t,s)\,\delta B(s)+\int_0^tD_s\left\{\mathcal{K}_g(Y)(t,s)\right\}\,ds\right) \\
&=Z\int_0^t Y(s)\, dX(s),
\end{align*}
which is (1). 

Proof of (2).  By  \citet[Lemma~3.11]{DNMBOP} we have \eqref{prule} with $D_s$ replaced with $D_{s,z}$. Hence by this equation and \eqref{extra}, 
\begin{align*}
\int_0^tZY(s)\,dX(s)&= 
\int_0^t\int_{\mathbb{R}}z\mathcal{K}_g(Y)(t,s)\left(Z+D_{s,z}\left\{Z\right\}\right)\,\widetilde{N}(\delta z,\delta s) \\
&\qquad+\int_0^t\int_{\mathbb{R}}zD_{s,z}\left\{\mathcal{K}_g(Y)(t,s)\right\}\left(Z+D_{s,z}\left\{Z\right\}\right)\,\widetilde{N}(\delta z,\delta s) \\
&\qquad-\int_0^t\int_{\mathbb{R}}zD_{s,z}\left\{\mathcal{K}_g(Y)(t,s)\right\}D_{s,z}\left\{Z\right\}\,\widetilde{N}(\delta z,\delta s) \\
&\qquad+\int_0^t\int_{\mathbb{R}}z\left(ZD_{s,z}\left\{\mathcal{K}_g(Y)(t,s)\right\}+D_{s,z}\left\{Z\right\}\mathcal{K}_g(Y)(t,s)\right)\,\ell(dz)\,ds
\end{align*} 
Invoking Proposition~\ref{prop4} 
 gives
\begin{align*}
\int_0^tZY(s)\,dX(s)&=Z\int_0^t\int_{\mathbb{R}}z\mathcal{K}_g(Y)(t,s)\,\widetilde{N}(\delta z,\delta s)-\int_0^t\int_{\mathbb{R}}z\mathcal{K}_g(Y)(t,s)D_{s,z}\left\{Z\right\}\,\ell(dz)\,ds \\
&\qquad+Z\int_0^t\int_{\mathbb{R}}zD_{s,z}\left\{\mathcal{K}_g(Y)(t,s)\right\}\,\widetilde{N}(\delta z,\delta s)\\
&\qquad-\int_0^t\int_{\mathbb{R}}
zD_{s,z}\left\{\mathcal{K}_g(Y)(t,s)\right\}D_{s,z}\left\{Z\right\}\,\ell(dz)\,ds \\
&\qquad-\int_0^t\int_{\mathbb{R}}zD_{s,z}\left\{\mathcal{K}_g(Y)(t,s)\right\}D_{s,z}\left\{Z\right\}\,\widetilde{N}(\delta z,\delta s) \\
&\qquad+\int_0^t\int_{\mathbb{R}}z\left(ZD_{s,z}\left\{\mathcal{K}_g(Y)(t,s)\right\}+D_{s,z}\left\{Z\right\}\mathcal{K}_g(Y)(t,s)\right)\,\ell(dz)\,ds\,.
\end{align*}
Hence, the result follows by collecting terms.
\end{proof}
We note that  \eqref{browcase} is a classical rule of calculation, which,  by Proposition    \ref{ibp} (2), is generally only satisfied in the Gaussian case. 
But \eqref{browcase}  is also satisfied in the pure--jump case when the integrand $Y$  is deterministic since we then have     $D_{s,z}\left\{\mathcal{K}_g(Y)(t,s)\right\}=0$.

Let us consider the integration of \lq\lq simple\rq\rq\   integrands.  Suppose that  $0\leq t_0<t_1<t_2<\ldots<t_n\leq t$ is a partition of the interval $[0,t]$ and let $Z_i$ be bounded random variables such that $s\mapsto Z_i\mathbb{I}_{\{t_i\leq s<t_{i+1}\}}$ belongs to $\mathcal{I}^X(0,t)$ respectively to $\widetilde{\mathcal{I}}^X(0,t)$.  Then we call 
$$
Y(s)=\sum_{i=0}^{n-1}Z_i\mathbb{I}_{\{t_i\leq s<t_{i+1}\}}\,
$$
a simple  integrand.  Since $s\mapsto \mathbb{I}_{(t_i\leq s<t_{i+1}\}}$ is deterministic it follows from Proposition~\ref{ibp}, equation \eqref{ex1} and linearity of the integral that 
\begin{equation}
\int_0^tY(s)\,dX(s)=\sum_{i=0}^{n-1}Z_i\Delta X(t_i)\,,
\end{equation}
with $\Delta X(t_i)=X(t_{i+1})-X(t_i)$.  Thus we see that the integral coincides with the simple integral on simple (non-adapted)  integrands.  We remark that this property generally only holds for {\it deterministic} simple integrands for the integral defined in \cite[p.\ 773]{AMN}.

We next show that our integral shares the {\it localization} property, which is also one of the basic properties of the It\^o integral.
\begin{proposition}
Suppose that $Y(s)=0, a.e.\ s\leq t$, $a.s.$ Then, under Assumptions  A or   B,  
$$
\int_{0}^tY(s)\,dX(s)=0\,, a.s.\,,
$$
\end{proposition}
\begin{proof}
Since $Y(s)=0$ for $a.e.\ s\leq t$, $a.s.$, we have that
$$
\int_s^t (Y(u)-Y(s))\,g(du,s)=0\,,
$$ 
for $a.e.\ s\leq t$ $a.s.$ Thus, $\mathcal{K}_g(Y)(t,s)\sigma(s)=0$ for $a.e.\ s\leq t$, $a.s.$ Then, by \citet[Proposition 1.3.6 ]{N} we find that 
$$
\int_0^t\mathcal{K}_g(Y)(t,s)\sigma(s)\,\delta B(s)=0,\,\quad a.s.
$$
 Furthermore, we have $D_s\mathcal{K}_g(Y)(t,s)=0$ and $D_{s,z}\mathcal{K}_g(Y)(t,s)=0$, except on sets of measure zero. Hence,
using the property of Lebesgue integrations, we find that 
$$
\int_0^t \left(D_s\mathcal{K}_g(Y)(t,s)+\int_{\mathbb{R}}zD_{s,z}\mathcal{K}_g(Y)(t,s)\,\ell(dz)\sigma(s)\right)\,ds=0,\, \quad a.s.
$$
Finally, by making use of Corollary 3.9 in~\cite{DNMBOP}, we find that
$$
\int_0^t\int_{\mathbb{R}}z\left(\mathcal{K}_g(Y)(t,s)+D_{s,z}\left\{\mathcal{K}_g(Y)(t,s)\right\}\right)\sigma(s)\,\widetilde{N}(\delta z,\delta s)\, { = 0}.
$$
Hence, the result follows. 
\end{proof}

\subsection{Applications}
\subsubsection{Deterministic integrands}
In this subsection we focus on deterministic integrands on the form $Y(s)=h(t,s)$, with $s\mapsto h(t,s)$ being a measurable function on $[0,t]$ such that $\mathcal{K}_g(h)(t,s)$ is well-defined (that is, $h(t,u)-h(t,s)$ is integrable with respect to $g(du,s)$).  Since we know that in this case the Malliavin derivatives in  \eqref{defbr} and \eqref{YdXJ} become zero, we have that 
\begin{align*}
\int_0^th(t,s)\,dX(s)
&=\int_0^t\mathcal{K}_g(h)(t,s)\sigma(s)\,dL(s)\, ,
\end{align*}
where we recall that $L= B$ or $L= \Lp$ depending on which Assumption we are working under. 
Thus, the integral of a deterministic function with respect to a \V process gives again a \V process. 

If we consider shift-kernel functions $h, g$, i.e., $h(t,s)=h(t-s), g(t,s)=g(t-s)$, we find that $\mathcal{K}_g(h)(t,s)$ depends on $s$ and $t$ only through $(t-s)$,
since from \eqref{stateq}
$$
\mathcal{K}_g(h)(t,s)=h(t-s)g(t-s)+\int_0^{t-s}\left(h(t-s-u)-h(t-s)\right)\,g(du)\,.
$$

Let us look at a simple example: Consider $h(u)=\exp(-\alpha u)$, $g(u)=\exp(-\beta u)$ for $\alpha, \beta > 0$ with  $\alpha\neq\beta$. Then, after a straightforward integration,
$$
\mathcal{K}_g(h)(t,s)=\frac{\alpha}{\alpha-\beta} e^{-\alpha(t-s)}-\frac{\beta}{\alpha-\beta} e^{-\beta(t-s)}
$$
and we find
\begin{align*}
\int_{0}^te^{-\alpha(t-s)}\,dX(s)&=\int_{0}^t\left\{\frac{\alpha}{\alpha-\beta} e^{-\alpha(t-s)}-\frac{\beta}{\alpha-\beta} e^{-\beta(t-s)}\right\}\sigma(s)\,dL(s)\,.
\end{align*}
If $\alpha=\beta$, we trace back and find
$$
\int_{0}^te^{-\alpha(t-s)}\,dX(s)=\int_{0}^t e^{-\alpha(t-s)}(1-\alpha(t-s))\sigma(s)\,dL(s)\,.
$$
\begin{rem}
This type of integral kernel appears in connection with stochastic volatility modelling for  commodity prices  based on the Schwartz dynamics and the Barndorff-Nielsen and Shephard model, see \cite{Benth-MF}.
\end{rem}
 Moreover, as the next Proposition shows, the integral of $h$ with respect to $X$ is also the solution of an  \vv--driven Ornstein-Uhlenbeck process:
\begin{proposition}\label{PropDetInt}
Assume that for all $s\in \R_+$ the mapping $(s,\infty)\ni u\mapsto e^{-\alpha(u-s)}-1$ is integrable with respect to $g(du,s)$, and   
\begin{align}\label{conv}
\lim_{t\downarrow s}(1-\exp(-\alpha(t-s)))g(t,s)=0\,.
\end{align}
Assume in addition that $t\mapsto \sigma(t)$ is pathwise locally bounded $a.s.$ Finally, assume that 
\begin{equation}\label{gass}
\int_0^t\int_u^t (g(s,u))^2\, ds\,du <\infty\, .
\end{equation}
Then, $s\mapsto e^{-\alpha(t-s)}$ is in $\mathcal{I}^X(0,t)$ for all $t>0$ when Assumption A is satisfied (or in $\widetilde{\mathcal{I}}^X(0,t)$ for all $t\leq T$ when Assumption B is satisfied) and  
$Y(t)=\int_0^t\exp(-\alpha(t-s))\,dX(s)$ solves the \vv--driven Ornstein--Uhlenbeck process
\begin{align}\label{OUSDE}
Y(t)=-\alpha\int_0^tY(s)\,ds+X(t)\,.
\end{align}
\end{proposition}
\begin{proof} By Assumption,  (1) in Definition~\ref{defbm} (resp.\ in Definition~\ref{defpj}) is satisfied and since the integrand is deterministic this implies that  $s\mapsto e^{-\alpha(t-s)}$ is in $\mathcal{I}^X(0,t)$  (resp.\ in $\widetilde{\mathcal{I}}^X(0,t)$ for all $t\leq T$).

Denote for simplicity $h(u)=\exp(-\alpha u)$, and we deduce from the definition of integration with respect to $X$ that
\begin{align*}
Y(t)+\alpha\int_0^tY(s)\,ds&=\int_0^th(t-u)\,dX(u)+\alpha\int_0^t\int_0^sh(s-u)\,dX(u)\,ds \\
&=\int_0^t\mathcal{K}_g(h)(t,u)\sigma(u)\,d L(u)+\alpha\int_0^t\int_0^s\mathcal{K}_g(h)(s,u)\sigma(u)\,d L(u)\,ds, 
\end{align*}
where we recall that $L$ is either $B$ or $\Lp$. 
Appealing to the stochastic Fubini Theorem  cf.\ details below   we find
\begin{align*}
Y(t)+\alpha\int_0^tY(s)\,ds
&=\int_0^t\left(\mathcal{K}_g(h)(t,u)+\alpha\int_u^t\mathcal{K}_g(h)(s,u)\,ds\right)\sigma(u)\,dL(u).
\end{align*}
Now, we have that 
\begin{multline*}
\mathcal{K}_g(h)(t,s) +\alpha\int_s^t\mathcal{K}_g(h)(u,s)\,du =e^{-\alpha(t-s)}g(t,s)+\int_s^t\left(e^{-\alpha(t-u)}-e^{-\alpha(t-s)}\right)\,g(du,s) \\
+\alpha\int_s^t\left(e^{-\alpha(u-s)}g(u,s)+\int_s^u\left(e^{-\alpha(u-v)}-e^{-\alpha(u-s)}\right)\,g(dv,s)\right)\,du\,.
\end{multline*}
But, by the Fubini theorem, we find
\begin{align*}
\alpha\int_s^t\int_s^u&\left(e^{-\alpha(u-v)}-e^{-\alpha(u-s)}\right)\,g(dv,s)\,du \\
 &=\alpha\int_s^t\int_v^t\left(e^{-\alpha(u-v)}-e^{-\alpha(u-s)}\right)\,du\,g(dv,s) \\
&=-\int_s^t \left(e^{-\alpha(t-v)}-e^{-\alpha(t-s)}\right)\,g(dv,s)+\int_s^t\left(1-e^{-\alpha(v-s)}\right)\,g(dv,s)\,.
\end{align*}
After doing an integration--by--parts and applying the limit condition on $g$, we get
\begin{equation}\label{yetanother}
\int_s^t\left(1-e^{-\alpha(v-s)}\right)\,g(dv,s)=(1-e^{-\alpha(t-s)})g(t,s)-\alpha\int_s^te^{-\alpha(v-s)}g(v,s)\,dv\,.
\end{equation}
This gives that 
\begin{align*}
\mathcal{K}_g(h)(t,s)&+\alpha\int_s^t\mathcal{K}_g(h)(u,s)\,du=g(t,s),
\end{align*}
and the proposition is proved except for the verification of the Fubini theorem, which is relegated to the following Lemma. 
\end{proof}
To justify the use of stochastic Fubini above  we have to verify the following condition from~\citet[Theorem IV.65]{P}.
\begin{lem}
Using the same notation as in the proof of Proposition \ref{PropDetInt},
\begin{equation}\label{Kgprop}
\int_0^t  \int_u^t (\mathcal{K}_g(h)(s,u))^2 \, ds\, \sigma^2(u)\, du <\infty, \quad a.s.
\end{equation}
\end{lem}
\begin{proof}
Since $u\mapsto\sigma(u)$ is locally bounded we may and do assume that $\sigma(u)=1$. By definition of $\mathcal{K}_g(h)$ and \eqref{yetanother}, 
\begin{align*}
\mathcal{K}_g(h)(s,u)&= h(u)g(s,u) + \int_u^s (h(v)- h(u))\, g(dv,u)\\
& = h(u)g(s,u) +e^{-\alpha u}\int_u^s (e^{-\alpha(v-u)} -1)\, g(dv,u)\\
&= h(u)g(s,u)  -h(u)\{(1-e^{-\alpha(s-u)}) g(s,u) - \alpha\int_u^s h(v-u)g(v,u)\, dv\}  \\
&= h(s) g(s,u)+ \alpha\int_u^s h(v)g(v,u)\, dv\, .
\end{align*}
Thus, it suffices to verify \eqref{Kgprop} with $\mathcal{K}_g(h)$ replaced by the two terms on the right-hand side. 
Since $h$ is bounded we have by \eqref{gass} that  
\begin{equation*}
\int_0^t  \int_u^t (h(s)g(s,u))^2 \, ds\, du <\infty\, .
\end{equation*}
Since $h(v)$ is bounded by $1$ for $v\geq 0$ we have by   Jensen's inequality that 
\begin{align*}
&\int_0^t \int_u^t \left(\int_u^sh(v)g(v,u)\, dv\right)^2 \, ds\, du 
 \leq \int_0^t \int_u^t \int_u^s(g(v,u))^2(s-u)\, dv\, ds\, du
\\
&\qquad \leq t\int_0^t  \int_u^t \int_u^t(g(v,u))^2\, dv\, ds\,  du \leq t^2\int_0^t  \int_u^t(g(v,u))^2\, dv  du <\infty
\end{align*}
which gives \eqref{Kgprop}. 
\end{proof}
We notice that a sufficient condition for the limit condition in the proposition to hold is that the limit of $g(t,s)$ exists as $t\downarrow s$. 

\begin{ex}
Recall the  example from turbulence: 
$$
g(t,s)=g(t-s) =(t-s)^{\nu-1}\exp(-\lambda (t-s))
$$
for $\nu>\frac{1}{2}$. This gives a singularity at $s=t$ when $\nu \in (\frac{1}{2},1)$. However, appealing to the L'Hopital rule, it is easily verified that
the condition \eqref{conv} is satisfied.
Hence, the condition on the limit behaviour of $g$ is satisfied for all  $\nu>\frac{1}{2}$.  Similarly, for fractional Brownian motion with kernel  \eqref{fbm1} or \eqref{fbm2} there is a singularity at $t=s$ when $H\in (0,\frac{1}{2})$ but the same kind of argument shows that the limit behaviour of $g$ is satisfied for all $H\in (0,1)$. 
\end{ex}

Note that in general, the solution $Y(t)$ of the \vv--driven OU process in \eqref{OUSDE} is {\it not} in a 
shift-kernel-type 
form. Since $g$ may not be a shift-kernel, we also get that $Y(t)$ is a \vv\ process with a  kernel function which is not of shift-kernel type. If $g$ is  a shift-kernel, then also $Y(t)$ will be a (semi) stationary \V (or more precisely an \lss) process. In \cite{BSZ-02} one has suggested Ornstein--Uhlenbeck processes driven by fractional Brownian motion as a model for the dynamics of temperature. Our analysis here links their results to \LSS processes. See also \citet{Barndorff-Basse11} for a general treatment of so-called quasi Ornstein-Uhlenbeck  processes and references to the literature in this field.

\subsubsection{A Volterra process as integrand}\label{volint}
It is natural in the present
context to consider the question of the characteristics of integrals of Volterra processes with
respect to Volterra processes. We shall not aim here to give a comprehensive discussion of
this point, but will restrict to an illustration in the \VB\ case.
 This will be connected to chaos expansions with respect to $X$ and, as we shall see, this \lq\lq double integral\rq\rq\  with respect to $X$ gives an element in second chaos, as expected.

 Throughout this section we work under Assumption A and $X$ is given by \eqref{def-bss}. Further assume that $\sigma(t)=1$, i.e., there is no stochastic volatility, 
and that $s\mapsto g(t,s)$ is a measurable function on $[0, t]$ where $g(s,s)$ is well-defined for all $s \in [0, t]$. 
Consider the integrand 
$$
Y(s)=\int_{0}^sh(s,v)\,dX(v)\,,
$$
for some sufficiently regular function $h\in\mathcal{I}^X(0,s)$.  We want to compute the integral
$$
\int_0^tY(s)\,dX(s)\,
$$
by using our theory for stochastic integration.

By Definition~\ref{defbm}  we find that 
$$
\int_0^tY(s)\,dX(s)=\int_0^t\mathcal{K}_g(Y)(t,s)\,\delta B(s)+\int_0^tD_s\left\{\mathcal{K}_g(Y)(t,s)\right\}\,ds\,,
$$
where
$$
Y(s)=\int_0^s\mathcal{K}_g(h)(s,v)\,dB(v)\,.
$$
First, let us compute $\mathcal{K}_g(Y)(t,s)$: By definition we have
\begin{align*}
\mathcal{K}_g(Y)(t,s)&=Y(s)g(t,s)+\int_s^t(Y(u)-Y(s))\,g(du,s) \\
&=\int_0^s\mathcal{K}_g(h)(s,v)g(t,s)\,dB(v)
\\
&\qquad 
+\int_s^t  \left( \int_0^u\mathcal{K}_g(h)(u,v)\,dB(v)-\int_0^s\mathcal{K}_g{ (h)}(s,v)\,dB(v)\, \right) g(du,s) \\
&=\int_0^s\mathcal{K}_g(h)(s,v)g(t,s)\,dB(v)+\int_{s}^t\int_0^s\mathcal{K}_g(h)(u,v)-\mathcal{K}_g(h)(s,v)\,dB(v)\,g(du,s) \\
&\qquad+\int_{s}^t\int_s^u\mathcal{K}_g(h)(u,v)\,dB(v)\,g(du,s) \\
&=\int_0^s\mathcal{K}_g(h)(s,v)g(t,s)\,dB(v)+\int_0^s\int_s^t\mathcal{K}_g(h)(u,v)-\mathcal{K}_g(h)(s,v)\,g(du,s)\,dB(v) \\
&\qquad+\int_s^t\int_v^t\mathcal{K}_g(h)(u,v)\,g(du,s)\,dB(v)\,.
\end{align*}
In the last equality we applied the stochastic Fubini theorem, cf.\ e.g.\ \citet[Theorem IV.65]{P}, which holds subject to weak regularity conditions. Observe that applying the $\mathcal{K}_g$ operator to the function $s\mapsto\mathcal{K}_g(h)(s,v)$ for fixed $v$, yields
$$
\mathcal{K}_g(\mathcal{K}_g(h)(\cdot,v))(t,s)=\mathcal{K}_g(h)(s,v)g(t,s)+\int_s^t(\mathcal{K}_g(h)(u,v)-\mathcal{K}_g(h)(s,v))\,g(du,s)\,.
$$ 
Hence,
\begin{align*}
\mathcal{K}_g(Y)(t,s)&=\int_0^s\mathcal{K}_g(\mathcal{K}_g(h)(\cdot,v))(t,s)\,dB(v)+\int_s^t\int_v^t\mathcal{K}_g(h)(u,v)\,g(du,s)\,dB(v) \\
&=\int_0^t\widetilde{\mathcal{K}}_{g,h}(s,v,t)\,dB(v)\,,
\end{align*}
where we have introduced the definition
\begin{equation}
\widetilde{\mathcal{K}}_{g,h}(s,v,t)=\mathbb{I}_{\{v\leq s\}}\mathcal{K}_g(\mathcal{K}_g(h)(\cdot,v))(t,s)+\mathbb{I}_{\{v>s\}}
\int_v^t\mathcal{K}_g(h)(u,v)\,g(du,s)\,.
\end{equation}
By properties of the Malliavin derivative, see \citet[p.~25]{N}, we have for $s\leq t$
\begin{align*}
D_s\mathcal{K}_g(Y)(t,s)&=D_s\int_0^t\widetilde{\mathcal{K}}_{g,h}(s,v,t)\,dB(v)=\widetilde{\mathcal{K}}_{g,h}(s,s,t)\,.
\end{align*}
Hence, we find that (after using the stochastic Fubini theorem, see e.g.\  \citet[Proposition 2.6]{NZ88}) 
\begin{align*}
\int_0^tY(s)\,dX(s)&=\int_0^t\int_0^t\widetilde{\mathcal{K}}_{g,h}(s,v,t)\,dB(v)\,\delta B(s)+\int_0^t\widetilde{\mathcal{K}}_{g,h}(s,s,t)\,ds \\
&=\int_0^t\int_0^s\widetilde{\mathcal{K}}_{g,h}(s,v,t)\,dB(v)\,dB(s)+\int_0^t\int_s^t\widetilde{\mathcal{K}}_{g,h}(s,v,t)\,dB(v)\,\delta B(s) \\
&\qquad+\int_0^t\widetilde{\mathcal{K}}_{g,h}(s,s,t)\,ds \\
&=\int_0^t\int_0^s\widetilde{\mathcal{K}}_{g,h}(s,v,t)\,dB(v)\,dB(s)+\int_0^t\int_0^v\widetilde{\mathcal{K}}_{g,h}(s,v,t)\,dB(s)\,d B(v) \\
&\qquad+\int_0^t\widetilde{\mathcal{K}}_{g,h}(s,s,t)\,ds \\
&=\int_0^t\int_0^s\left(\widetilde{\mathcal{K}}_{g,h}(s,v,t)+\widetilde{\mathcal{K}}_{g,h}(v,s,t)\right)\,dB(v)\,dB(s)+\int_0^t\widetilde{\mathcal{K}}_{g,h}(s,s,t)\,ds\,.
\end{align*}
Hence, we have the explicit representation of an integral of a Volterra process $Y(t)$ with respect to another Volterra process $X(t)$ in terms of a sum of a second and zeroth order chaos.

 From the general result above we deduce a representation result for $\int_0^t X(s) dX(s)$ next. 
Observe that for $h=1$ we have $Y(t)=X(t)$. Furthermore, 
$
\mathcal{K}_g(1)(u,v)=g(u,v)$. 
Thus, as long as $u\mapsto g(u,v)$ is in the domain of $\mathcal{K}_g$, we find
$$
\mathcal{K}_g(\mathcal{K}_g(1)(\cdot,v))(t,s)=\mathcal{K}_g(g(\cdot,v))(t,s)=g(s,v)g(t,s)+\int_s^t(g(u,v)-g(s,v))\,g(du,s)\,,
$$
and
$$
\int_v^t\mathcal{K}_g(1)(u,v)\,g(du,s)=\int_v^tg(u,v)\,g(du,s)\,,
$$
and hence we 
have 
\begin{align*}
\widetilde{\mathcal{K}}_{g,1}(s,v,t)
&= \mathbb{I}_{\{v\leq s\}}\left(g(s,v)g(t,s)+\int_s^t(g(u,v)-g(s,v))\,g(du,s)\right)
+
\mathbb{I}_{\{v>s\}} \int_v^tg(u,v)\,g(du,s)\\
&= \mathbb{I}_{\{v\leq s\}}\left(g(s,v)g(s,s)+\int_s^tg(u,v)\,g(du,s)\right)
+
\mathbb{I}_{\{v>s\}} \int_v^tg(u,v)\,g(du,s),
\end{align*}
and $\widetilde{\mathcal{K}}_{g,1}(s,s,t)
= 
g^2(s,s)+\int_s^tg(u,s)g(du,s)$\,. 
Therefore, we
obtain
\begin{align*}
&\widetilde{\mathcal{K}}_{g,1}(s,v,t)+\widetilde{\mathcal{K}}_{g,1}(v,s,t)\\
&= \mathbb{I}_{\{v\leq s\}}\left(g(s,v)g(s,s)+\int_s^tg(u,v)\,g(du,s)\right)
+
\mathbb{I}_{\{v>s\}} \int_v^tg(u,v)\,g(du,s)\\
&\qquad + \mathbb{I}_{\{s\leq v\}}\left(g(v,s)g(v,v)+\int_v^tg(u,s)\,g(du,v)\right)
+
\mathbb{I}_{\{s>v\}} \int_s^tg(u,s)\,g(du,v).
\end{align*}
Altogether, we get
\begin{align*}
\int_0^tX(s)\,dX(s) &= \int_0^t\int_0^s\left(\widetilde{\mathcal{K}}_{g,1}(s,v,t)+\widetilde{\mathcal{K}}_{g,1}(v,s,t)\right)\,dB(v)\,dB(s)+\int_0^t\widetilde{\mathcal{K}}_{g,1}(s,s,t)\,ds.
\end{align*}

Let us compare the previous result with $\frac12 X^2(t)$. By again using integration-by-parts in Proposition~\ref{prop2}, we have
\begin{align*}
X^2(t)&=\left(\int_0^tg(t,v)\,dB(v)\right)^2 
=\left(\int_0^tg(t,v)\,dB(v)\right)\left(\int_0^tg(t,s)\,dB(s)\right) \\
&=\int_0^tg(t,v)\left(\int_0^tg(t,s)\,dB(s)\right)\,\delta B(v)+\int_0^tg(t,v)D_v\left(\int_0^tg(t,s)\,dB(s)\right)\,dv \\
&=\int_0^tg(t,v)\int_0^vg(t,s)\,dB(s)\,dB(v)+\int_0^tg(t,v)\int_v^tg(t,s)\,dB(s)\,\delta B(v)+\int_0^tg^2(t,v)\,dv \\
&=\int_0^t\int_0^vg(t,v)g(t,s)\,dB(s)\,dB(v)+\int_0^t\int_0^vg(t,v)g(t,s)\,dB(s)\,dB(v)+\int_0^tg^2(t,v)\,dv\,.
\end{align*}
Here, we have used the Fubini theorem for double stochastic Brownian integrals once again. 
Hence, we have
\begin{align*}
  \frac12 X^2(t)
= \int_0^t\int_0^sg(t,s)g(t,v)\,dB(v)\,dB(s)+\frac12 \int_0^tg^2(t,s)\,ds\,.
\end{align*}

\section{Conclusion and outlook}\label{SectConcl}
In this paper we have developed a stochastic integration theory with respect to volatility modulated L\'{e}vy--driven Volterra (\V) processes. Our results extend previous work in the literature to allow for stochastic volatility and pure jump processes in the integrator. So far, we have treated the case of a \V process driven by a Brownian motion and  a pure jump process separately. In future research, we plan to address the problem of  accounting for a Brownian motion and a jump process simultaneously. 

Moreover, the results in this paper can be regarded as the key building block for developing a general integration theory for \V processes defined on the entire real line, i.e.~for processes of type 
\begin{align}\label{gencase}
X(t) = \int_{-\infty}^t g(t,s ) \sigma(s) d L(s), \qquad t \in \R.
\end{align}
Note that  stochastic  integrals of type \eqref{gencase} can be defined using the recent theory by 
\citet{BassePedersen2010}, who studied stochastic integration  with respect to so-called increment semimartingales on the real line.

\section*{Acknowledgement}
It is a pleasure to thank Giulia Di Nunno for interesting discussions.  F.~E.~Benth is grateful for the financial support from the project \lq\lq Energy Markets: Modelling, Optimization and Simulation\rq\rq (EMMOS) funded by the Norwegian Research Council under the grant eVita/205328. Financial support by the Center for Research in Econometric
Analysis of Time Series, CREATES, funded by the Danish National Research Foundation is gratefully
acknowledged by  A.~E.~D.~Veraart.

\bibliographystyle{agsm}
\bibliography{IntegrationBib}

\end{document}